\documentclass[12pt,a4paper]{article}
\usepackage[utf8]{inputenc}
\usepackage[style=ieee,sorting=nyvt]{biblatex} 
\makeatletter
\def\blfootnote{\xdef\@thefnmark{}\@footnotetext}
\makeatother
\usepackage[margin=2.5cm]{geometry}
\usepackage{amsmath,amsfonts,amssymb,amsthm}
\usepackage{bm} 
\usepackage{graphicx}
\usepackage{epstopdf}
\usepackage{soul}
\usepackage{xcolor}
\usepackage{pinlabel}
\usepackage{fourier}\usepackage{microtype}
\usepackage{caption}\usepackage{verbatim}
\renewcommand\footnotemark{}

\usepackage{soul} 
\usepackage[bookmarks=false]{hyperref}
\hypersetup{
    colorlinks=true,
    linkcolor=blue,
    filecolor=magenta,      
    urlcolor=cyan,
    }
\newcommand{\N}{\mathbb N}

\newcommand{\CP}{\mathbb{CP}}

\renewcommand{\L}{\mathcal L}

\newcommand{\Z}{\mathbb Z}

\renewcommand{\S}{\mathcal S}
\newcommand{\T}{\mathcal T}
\renewcommand{\H}{\mathcal H}

\newcommand{\al}{\alpha}
\newcommand{\be}{\beta}
\newcommand{\de}{\delta}
\newcommand{\ep}{\varepsilon}
\newcommand{\si}{\sigma}
\newcommand{\ga}{\gamma}
\newcommand{\De}{\Delta}
\newcommand{\Si}{\Sigma}
\newcommand{\Ga}{\Gamma}

\newcommand{\ula}{\bm{\lambda}}
\newcommand{\umu}{\bm{\mu}}
\newcommand{\la}{\lambda}

\newcommand{\x}{\times}
\newcommand{\h}{\mathfrak{h}}
\newcommand{\id}{id} 
\newcommand{\hra}{\hookrightarrow}
\newcommand{\del}{\partial}

\renewcommand{\int}{\rm int}
\newcommand{\co}{\thinspace\colon}
\newcommand{\lra}{\longrightarrow}
\newcommand*\sm[1]{\left(\begin{smallmatrix}#1\end{smallmatrix}\right)}
\newcommand*\mat[1]{\begin{pmatrix}#1\end{pmatrix}}

\DeclareMathOperator{\fr}{fr}
\DeclareMathOperator{\Map}{Map}

\newcommand{\lk}{lk}

\newtheorem{thm}{Theorem}[section]
\newtheorem{lemma}[thm]{Lemma}

\newtheorem{prop}[thm]{Proposition}

\theoremstyle{definition}

\newtheorem{defn}[thm]{Definition}
\newtheorem{rmk}[thm]{Remark}

\newtheorem*{ack}{Acknowledgments}

\newtheorem{exa}[thm]{Example}
\numberwithin{equation}{section}
\newcommand{\Addresses}{{% additional braces for segregating \footnotesize
  \bigskip
  \footnotesize
  \textsc{Dipartimento di Matematica, Largo Bruno Pontecorvo 5, 56127 Pisa, Italy}\par\nopagebreak
  \textit{E-mail addresses:}\ \texttt{paolo.lisca@unipi.it, andrea.parma94@gmail.com}
}}

\title{Horizontal decompositions, I}
\begin{document}
\author{Paolo Lisca \and Andrea Parma}
\maketitle	

\begin{abstract}
We show that every smooth, closed, orientable 4-manifold $X$ admits a special kind of handlebody decomposition 
that we call {\em horizontal}. We classify the closed $4$-manifolds with the simplest 
horizontal decompositions and we describe all such decompositions of $\CP^2$, showing that they give rise 
to infinitely many of the known embeddings of rational homology balls in the complex projective plane. 
\end{abstract} 


\section{Introduction}\label{s:intro} 
Let $p>q\geq 0$ be coprime integers and $B_{p,q}$ the rational homology ball smoothing of the quotient singularity $\frac{1}{p^2}(pq-1,1)$ used by Fintushel and Stern in the rational blow-down construction~\cite{FS97}. We were led to the results to this paper by the realization that the smooth embeddings  $B_{p,q}\hra\CP^2$ constructed in~\cite{LP20} were obtained using certain special handlebody decompositions of $\CP^2$ and that every smooth, closed, orientable $4$-manifold admits such  decompositions. The purpose of this paper is to define horizontal decompositions, prove the general existence result just mentioned, study the simplest cases and illustrate their potential applications by showing how to recover infinitely many of the known embeddings of the $B_{p,q}$'s into $\CP^2$ in a ``systematic'' way. 

Let $X\co \del_- X\to \del_+ X$ be a smooth, oriented, $4$-dimensional cobordism, i.e.~a smooth, oriented, compact $4$-manifold with oriented boundary $\del X = \del_+ X\cup -\del_- X$. We always assume that $\del_- X$ and $\del_+ X$ are non-empty and connected. By e.g.~\cite[Proposition~4.2.13]{GS99}, $X$ admits a handlebody decomposition relative to $\del_- X$ without $0$- nor $4$-handles. Suppose such a decomposition has $u$ $1$-handles and $h$ $3$-handles. The $1$-handles give a cobordism $\del_- X\to \del_-^u X := \del_- X\#^u S^1\x S^2$, and the $3$-handles a cobordism $\del_+^h X :=\del_+ X\#^h S^1\x S^2\to \del_+ X$. The cobordism $\del_-^u X\to \del_+^h X$ given by the $2$-handles determines and is determined by a framed link $L\subset \del^u_- X$. We denote it by $X_L$. 
\begin{defn}\label{d:horizL}
Let $Y$ be a closed, oriented $3$-manifold and $L=\cup_{i=1}^k L_i\subset Y$ a 
$k$-component framed link. We say that $L$ is a {\em horizontal link of type $g$} if $g>0$ and there is a 
decomposition $Y = H_g \cup \Si_g\x [0,1]\cup H'_g$, where $H_g$, $H'_g$ are three-dimensional genus-$g$ handlebodies, 
and real numbers $0<t_1<\cdots<t_k<1$ such that, for each $i=1,\ldots,k$: 
\begin{itemize}
\item
$L_i$ is a homotopically non-trivial simple closed curve on 
$\mathfrak{S}_i:=\Sigma_g\x \{t_i\}$;
\item
|$\fr(L_i) - \fr_{\mathfrak{S}_i}(L_i)| = 1$, where 
$\fr(L_i)$ is the given framing of $L_i$ and $\fr_{\mathfrak{S}_i}(L_i)$ the framing  induced on $L_i$ by $\mathfrak{S}_i$.
\end{itemize} 
We call the difference $\fr(L_i) - \fr_{\mathfrak{S}_i}(L_i)$ the {\em relative framing}  of $L_i$. 
\end{defn}

\begin{defn}\label{d:hfl-hhd}
When $L\subset \del^u_- X$ is a horizontal link of type $g$ 
with $\ell = |L|$ components and $X_L\co \del_-^u X\to \del_+^h X$, 
we call a handlebody decomposition of an oriented cobordism 
$X\co \del_- X\to \del_+ X$ 
as above a {\em horizontal decomposition} of $X$ of type $(g,u,\ell,h)$, or simply 
a {\em $(g,u,\ell,h)$-decomposition} of $X$. 
We say that a smooth, closed 4-manifold $\hat X$ has a horizontal 
decomposition of type $(g,u,\ell,h)$ 
if $X:=\hat X\setminus\{B^4\cup B^4\}$ has one as a cobordism $S^3\to S^3$.
\end{defn} 

Note that $g\geq b_1(\del_-^u X) = u + b_1(\del_- X)\geq u$.  
Moreover, it is not difficult to check that if a link $L\subset \del_-^u X$ is horizontal 
with respect to a genus-$g$ Heegaard splitting  
$H_g\cup_{\Si_g}\cup H'_g = \del_-^u X$, then $\del_+^h X$ has a genus-$g$ Heegaard 
splitting obtained by cutting along $\Si_g$ and regluing via a diffeomorphism -- see Lemma~\ref{l:heegaard}. 
In particular, the Heegaard genus of $\del_+^h X$ is at most $g$. 
As above, the same number cannot be less than $h$.  
This shows that {\em for every horizontal decomposition of type $(g,u,\ell,h)$ 
the inequality $g\geq\max(u,h)$ holds}.

\begin{exa}\label{exa:exa1}
The following simple example illustrates the above definitions. View the $\langle 0\rangle$-labelled unknot in Figure~\ref{f:s2xs2} as a surgery description of $S^1\x S^2$. The torus $T$ visible in Figure~\ref{f:s2xs2} is a Heegaard surface, the framed link $\L$ given by the three framed knots is framed isotopic to several type-$1$ horizontal links $L$, and in each case $X_L\co S^1\x S^2\to S^3$.  
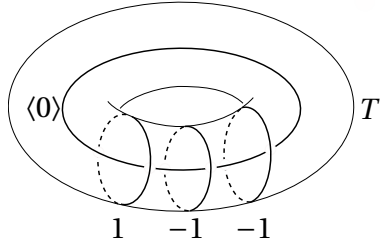
\begin{figure}[ht]
\labellist
\hair 2pt
\pinlabel $T$ at 345 100
\pinlabel $\langle 0\rangle$ at 34 100
\pinlabel $1$ at 105 -15
\pinlabel $-1$ at 170 -15
\pinlabel $-1$ at 235 -15
\endlabellist
\centering
\includegraphics[scale=0.4]{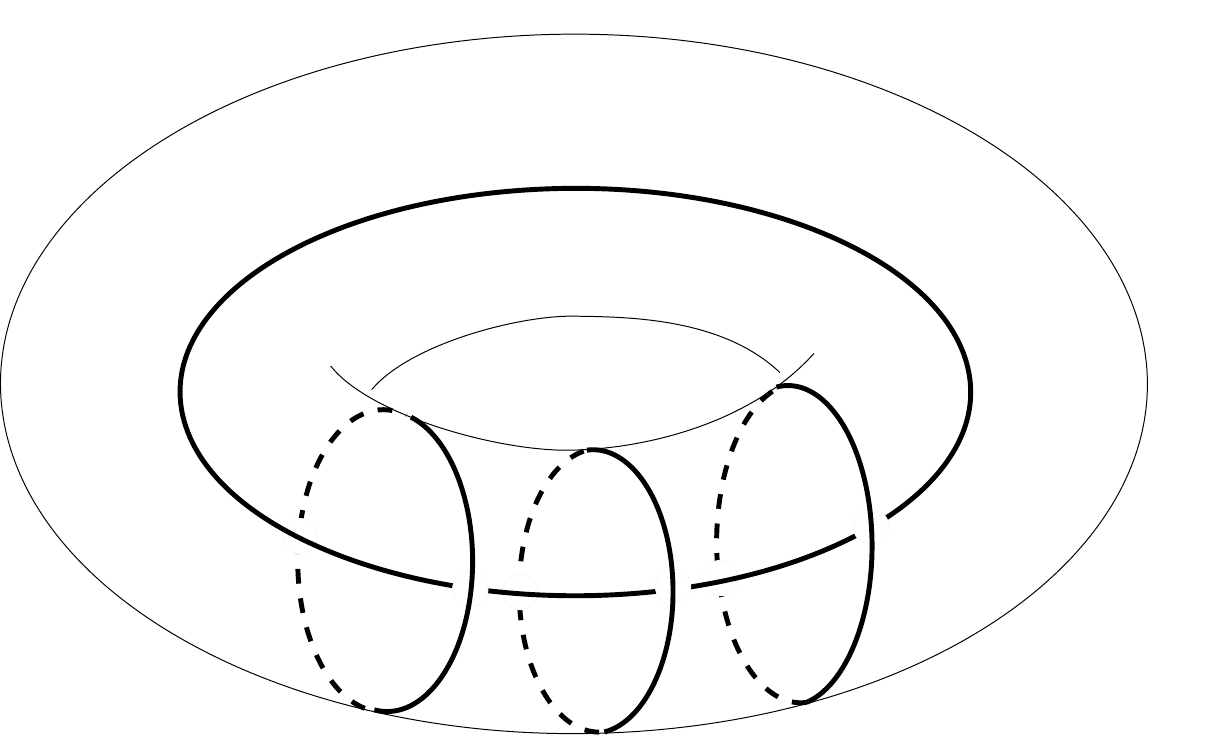}
\vspace*{0.2cm}
\caption{A framed link giving $(1,1,2,0)$-decompositions of $S^2\x S^2$}
\label{f:s2xs2}
\end{figure}
By~\cite{Ce68} and~\cite{LP72}, there is a canonical way to obtain a closed $4$-manifold $\widehat X_L$ gluing $S^1\x D^3$ to $X_L$ along $S^1\x S^2$ and a $4$-ball along $S^3$. An easy application of Kirby calculus shows that  $\widehat X_L = S^2\x S^2$. 
\end{exa} 

The following establishes a basic existence result for horizontal decompositions.

\begin{thm}\label{t:exist} 
Let $X\co \del_- X\to \del_+ X$ be a smooth, oriented, connected, $4$-dimensional cobordism 
such that $\del_- X$ and $\del_+ X$ are non-empty and connected. Then, $X$ 
admits a horizontal decomposition of type $(u+g_H(\del_- X),u,\ell,h)$, where 
$g_H(\del_- X)$ is the Heegaard genus of $\del_- X$, and such that 
each component of the associated horizontal link is non-separating in the Heegaard surface. 
\end{thm} 

\begin{rmk}\label{r:tunnel} 
The referee has pointed out that the existence of horizontal decompositions also 
follows from the fact that each framed link is horizontal with respect to some Heegaard 
splitting. At the end of Section~\ref{s:exist} we sketch a proof of this fact following the referee's 
suggestion.
\end{rmk}

We prove Theorem~\ref{t:exist} by showing that any handlebody decomposition of $X$ can be made horizontal by 
adding $(1,2)$-cancelling pairs and sliding $2$-handles -- see Section~\ref{s:exist}. 
Our proof bears some resemblance with John Harer's argument showing that any 4-dimensional 2-handlebody 
admits an achiral Lefschetz fibration with bounded fibers over $D^2$ (\cite{Ha79}, see also~\cite[Theorem~7]{EF06}). 
Theorem~\ref{t:exist} may also be compared to the main result in~\cite{EF06}, where 
Etnyre and Fuller combine Harer's result with Giroux's correspondence~\cite{Gi02} to prove that 
the complement of a smoothly embedded circle in a closed 4-manifold admits an achiral fibration over $S^2$.  
Nevertheless, the two results appear to us genuinely different, and in any case we do not rely on Giroux's correspondence. 

\begin{rmk}\label{r:genus} 
In view of Definition~\ref{d:hfl-hhd} and Theorem~\ref{t:exist}, 
it makes sense to define the {\em genus} of a smooth, oriented cobordism or 
closed 4-manifold $X$ as the smallest $g\geq 1$ such that $X$ admits a 
$(g,u,\ell,h)$-decomposition. It is easy to check 
that a cobordism (or closed 4-manifold) $X$ with genus $g(X)=1$ and first Betti number $b_1(X)=0$ must have 
finite cyclic fundamental group. This implies that there are many smooth, closed 4-manifolds $X$ with $b_1(X)=0$, 
$g(X)>1$ and $\chi(X) = 2$, which is the 
smallest possible value of the Euler characteristic under these assumptions.   
As kindly pointed out by the referee, simple examples are given by spun aspherical homology $3$-spheres~\cite{Go76, Su88}. 
\end{rmk}

In order to get the first classification results for cobordisms with horizontal decompositions we introduce some simplifying assumptions. 
A natural choice is to look at cobordisms $X\co S^3\to S^3$. 
The following result deals with the simplest cases of 
$(1,u,\ell,0)$-decompositions. In a forthcoming paper~\cite{LP22} we deal 
with the simplest $(1,u,\ell,1)$-decompositions. 
\begin{thm}\label{t:type10}
Let $X\co S^3\to S^3$ be an oriented cobordism with a $(1,u,\ell,0)$-decomposition. Then, after possibly reversing the orientation of $X$, we have 
\[
X \cong 
\begin{cases}
S^3\x [0,1] & \quad\text{if $\chi(X)=0$},\\
\CP^2\setminus\left(B^4\sqcup B^4\right) & \quad\text{if $\chi(X)=1$.}
\end{cases} 
\]
\end{thm} 

A key fact used in the proof of Theorem~\ref{t:type10} is that to a horizontal framed link $L\subset\Si_g\x [0,1]$ one can naturally associate an element of the mapping class group $\Map(\Si_g)$ factorized as a product of Dehn twists. In fact, suppose $L = \bigcup_{i=1}^k L_i$, with 
\[
L_i \subset \mathfrak{S}_i:=\Sigma_g\x \{t_i\},\qquad 0<t_1<\cdots<t_k<1.
\]
Since $\pi$ identifies each $\mathfrak{S}_i$ 
with $\Si_g$, one can associate to each component $L_i$ the Dehn twist 
$$
\tau_i := \tau_{\pi(L_i)}\in\Map(\Si_g),
$$ 
where $\pi\co\Si_g\x [0,1]\to\Si_g$ is the projection onto the first factor. We define the {\em factorization} of $L$ to be the $k$-tuple   
\[
F_L = (\tau_k^{\de_k}, \tau_{k-1}^{\de_{k-1}},\ldots, \tau_1^{\de_1}),
\] 
where the exponent $\de_i := \fr_{\mathfrak{S}_i}(L_i) - \fr(L_i)\in\{\pm 1\}$ is equal 
to {\em minus} the relative framing  of $L_i$ for each $i$. 
The product $m_L = \tau_k^{\de_k}\cdots\tau_1^{\de_1}\in\Map(\Si_g)$ will be called the {\em monodromy} of $L$. 
Sometimes, when the horizontal link $L$ is clear from the context, we will abuse 
of language and talk about the factorization or the monodromy of the decomposition 
of the associated cobordisms $X_L$ or $X$. 

\begin{exa}\label{exa:exa2}
The concept of factorization can be exemplified via the link $\L$ of Example~\ref{exa:exa1}. 
As we already observed, $\L$ is framed isotopic to several type-$1$ horizontal links 
$L = \cup_{i=1}^3 L_i\subset S^1\x S^2$. For example, we can arrange that the 1-framed unknot becomes the component 
$L_1\subset\Si_g\x\{t_1\}$ at the "lowest level". In that case the horizontal link $L$ has factorization $(\tau_\mu,\tau_\mu,\tau^{-1}_\mu)$, 
where $\mu\subset T$ is the meridian of the $0$-framed unknot. If we replace $\L$ with any sublink $\L'\subset\L$ consisting of two 
components of $\L$, depending on the choice of $\L'$ and the isotopies the corresponding horizontal 
link $L'$ has factorization $(\tau_\mu^{-1},\tau_\mu)$, $(\tau_\mu,\tau_\mu^{-1})$ or $(\tau_\mu,\tau_\mu)$. 
\end{exa} 

In the proof of Theorem~\ref{t:type10} we determine the factorizations of all 
$(1,u,\ell,0)$-decompositions of $\CP^2$ minus two disjoint balls, viewed as an 
oriented cobordism $S^3\to S^3$. 
We do that by relating such factorizations with the solutions of certain 
degree-two Diophantine equations. We exploit a certain structure of the solutions of 
those equations to show that any horizontal decomposition of $X$ can be simplified in 
a ``systematic'' way. For this approach to the proof of Theorem~\ref{t:type10} 
we took inspiration from Denis Auroux's paper~\cite{Au15}. 

As we mentioned above, we discovered horizontal decompositions while trying to construct smooth embeddings of 
the rational balls $B_{p,q}$ into $\CP^2$. Hence, it should not be surprising that the proof of 
Theorem~\ref{t:type10} yields such embeddings, leading to   
Theorem~\ref{t:disj-emb} below. In order to give some motivation we provide the following brief survey of 
results about embeddings of rational homology balls in $\CP^2$. 

It is well-known that $\del B_{p,q}$ is the lens space $L(p^2,pq-1)$. Moreover, $B_{p,p-q}$ 
is orientation-preserving diffeomorphic to $B_{p,q}$ by~\cite[Remark~2.8]{ES18}. 
Using results by Khodorovskiy~\cite{Kh13} it is not hard to show~\cite[Section~2.1]{ES18} that if the 
positive integers $p_1$, $p_2$ and $p_3$ form a {\em Markov triple}, that is 
$p_1^2+p_2^2+p_3^2 = 3 p_1 p_2 p_3$, then there is a symplectic embedding of the disjoint union 
$\cup_{i=1}^3 B_{p_i,q_i}$ into $\CP^2$, where $q_i=\pm 3p_j/p_k \bmod p_i$ with $\{i,j,k\}=\{1,2,3\}$. 
Conversely, Evans and Smith~\cite[Theorem~1.2]{ES18} generalized results of Hacking and Prokhorov~\cite{HP10} in complex algebraic geometry by showing that if $B_{p_i,q_i}\subset\CP^2$, $i=1,...,N$, is a collection of pairwise disjoint symplectic embeddings, then $N\leq 3$ and the $p_i$'s and the $q_i$'s 
satisfy some constraints which coincide with the ones above when $N=3$. 
Moreover, they showed that if $B_{p,q}\subset\CP^2$ is a symplectic embedding then 
$\CP^2\setminus B_{p,q}$ contains a disjoint union $B_{p',q'}\sqcup B_{p'',q''}$, therefore $p$ must belong 
to a Markov triple and divide $q^2+9$. Let \{$F_m\}_{m\in\Z}$ be the classical Fibonacci sequence defined by $F_{-1} = 1$, $F_0 = 0$ 
and $F_{m+1} = F_{m} + F_{m-1}$. Owens~\cite[Theorem~1]{Ow20} recently proved the existence of smooth 
embeddings $B_{F_{2m+1}, F_{2m-1}}\subset\CP^2$ for each $m\geq 1$. He also observed that $F_{2m+1}$ divides 
$F_{2m-1}^2+9$ only if $m=1$, so it follows from~\cite{ES18} that 
$B_{F_{2m+1}, F_{2m-1}}$ does not embed {\em symplectically} in $\CP^2$ for $m>1$. Moreover, Owens used 
Donaldson's Theorem~A~\cite{Do83} to show that it is not possible to embed in $\CP^2$ a disjoint union of two 
or more of the rational homology balls $B_{F_{2m+1}, F_{2m-1}}$ with $m\geq 1$. In~\cite{LP20} we extend Owens' 
family of smooth embeddings to a two-parameter family $\{B(k,m)\}\subset\{B_{p,q}\}$ such that $B(k,m)$ cannot 
be symplectically embedded in $\CP^2$. 
Moreover, in~\cite{LP20-2} we prove the non-existence of almost complex embeddings $B(k,m)\subset\CP^2$ 
without relying on~\cite{ES18}. 

The proof of Theorem~\ref{t:type10} shows that when $\chi(X)=1$  
most of the horizontal decompositions of $X$ contain one 1-handle and two 2-handles. 
As explained in Section~\ref{s:disj-emb}, this implies that many 
horizontal decompositions of $\CP^2$ yield a smooth embedding of a disjoint union of $B_{p,q}$'s into $\CP^2$. 
In Section~\ref{s:disj-emb} we use this fact to prove the following Theorem~\ref{t:disj-emb}. 
Note that $F_{2m+1}>0$ for each $m\in\Z$ and see Theorem~\ref{t:precise-disj-emb} 
for a more precise version of the statement. 
\begin{thm}\label{t:disj-emb} 
The $(1,1,2,0)$-decompositions of $\CP^2$ 
give rise to smooth, orientation-preserving embeddings into $\CP^2$ of 
the disjoint unions 
\[
B_{m+1}\sqcup B_m\quad\text{and}\quad
B'_{m+1}\sqcup B'_m
\]
\begin{comment} 
\begin{itemize}
\item 
$B_{F_{2m+1},F_{2m-1}}\cup -B_{F_{2m},F_{2m-1}}$,
\item 
$B_{F_{2m+1},F_{2m-1}}\cup -B_{F_{2m+2},F_{2m+1}}$, and 
\item 
$B_{F_{2m+1},F_{2m-3}}\cup B_{F_{2m-1},F_{2m-5}}$
\end{itemize}
\end{comment} 
for each $m\geq 0$ where $B_m = B_{F_{2m-1},F_{2m-5}}$ and $B'_m = (-1)^m B_{F_{m+1},F_{m}}$. 
\end{thm}

\begin{rmk} The rational balls $B_m$ form a subfamily of the $B_{p,q}$'s that embed symplectically in $\CP^2$. More precisely, for any 
$m \ge 0$ the triple $(1,F_{2m-1},F_{2m+1})$ is a Markov triple producing a symplectic embedding of $B^4 \cup B_m \cup B_{m+1}$. On the other hand, the balls $B'_{2k}=B_{F_{2k+1},F_{2k}} \cong B_{F_{2k+1},F_{2k+1}-F_{2k}}=B_{F_{2k+1},F_{2k-1}}$ are exactly the rational balls shown by Owens to embed smoothly but not symplectically.
\end{rmk}

In~\cite{LP22} we use horizontal decompositions to prove the existence of many more 
smooth embeddings of the rational balls $B_{p,q}$ into $\CP^2$. 

A few questions are naturally raised by our results. First of all, it is natural to wonder about the 
relationship between the various horizontal decompositions of a given cobordism. More precisely, one can ask 
whether there exists a uniqueness result for horizontal decompositions, perhaps up to the Hurwitz moves 
of Section~\ref{ss:hurwitz} combined with some kind of ``horizontal stabilizations''. Secondly, we could ask 
what are the possible generalizations of Theorem~\ref{t:type10}. For instance, we are 
planning to establish the analogue of Theorem~\ref{t:type10} for the simplest cobordisms 
$S^3\to S^3$ with a $(1,u,\ell,1)$-decomposition~\cite{LP22}. Finally, we expect that  
the techniques of the present paper can be applied to find new smooth embeddings of rational 
balls and/or $3$-manifolds in $\CP^2$ and/or other smooth $4$-manifolds with small Euler characteristic. 

The paper is organized as follows. In Section~\ref{s:exist} we prove Theorem~\ref{t:exist}, in 
Sections~\ref{s:type10-chi=0} and~\ref{s:type10-chi=1} we prove Theorem~\ref{t:type10} and in 
Section~\ref{s:disj-emb} we prove Theorem~\ref{t:precise-disj-emb}, 
which implies Theorem~\ref{t:disj-emb}. 

\begin{ack}
The authors would like to thank the referee for their great job and for several helpful comments. 
The present work is part of the MIUR-PRIN project 2017JZ2SW5. 
\end{ack}

\section{Proof of Theorem~\ref{t:exist}}\label{s:exist}

We refer the reader to~\cite[Chapter~4]{GS99} for basic facts about handlebody 
decompositions. The plan of the proof is as follows. We start from any handle decomposition of a smooth $4$-dimensional cobordism $X\co \del_- X\to \del_+ X$ and we modify it so that the framed link $L$ consisting of the attaching curves of the $2-$handles satisfies the conditions of Definition~\ref{d:hfl-hhd}. 

We start by recalling a few facts about the operation of addition of a $(1,2)$-cancelling pair. 
Let $\H$ be a handle decomposition of $X$ without $0$- nor $4$-handles, $u$ $1$-handles and $h$ $3$-handles. Then, attaching the $1$-handles gives a cobordism $\del_- X\to \del^u_- X$. 
A four-dimensional $(1,2)$-cancelling pair consists of a $1$-handle $\h$ and a $2$-handle 
$\h'$ such that the attaching sphere of $\h'$ intersects the belt sphere of $\h$ transversely in 
a single point. The introduction a $(1,2)$-cancelling pair does not alter 
$X\co \del_- X\to\del_+ X$ but turns $\H$ into a new handlebody decomposition $\H'$ of 
$X$ in such a way that the $1$-handles of $\H'$ give a subcobordism $X'\subset X$ from 
$\del_- X$ to the $3$-manifold $Y$ obtained from $\del^u_- X$ by removing the 
interior of two disjoint $3$-balls $B_1$ and $B_2$ and gluing to each other the corresponding boundary $2$-spheres via an orientation-reversing diffeomorphism.
In other words, $Y$ is diffeomorphic to $\del^{u+1}_- X$ and $X'\co \del_- X\to Y$. 
The cancelling $2-$handle is attached along any arbitrarily framed curve in 
$\del^u_- X\setminus (B_1\cup B_2)$ that connects two identified points of $\del B_1$ 
and $\del B_2$. Such a curve determines a component of the new framed link in $\del^{u+1}_- X$. 

Let now $\Si\subset \del^u_- X$ be a  Heegaard surface of genus 
$g_H(\del^u_- X) = u + g_H(\del_- X)$. Since we can always add a $(1,2)$-cancelling 
pair, from now on we assume $g>0$ without loss of generality. The complement 
of $\Si$ in $\del^u_- X$ is a disjoint union $H_g\cup H'_g$ of two genus-$g$ handlebodies. Moreover, 
the handlebodies $H_g$ and $H'_g$ are regular neighborhoods of graphs $G\subset H_g$ and $G'\subset H'_g$, and 
$\del^u_- X\setminus\{G\cup G'\}$ is diffeomorphic to the product of $\Si$ with an open interval. 
Up to framed isotopy we may assume $L$ disjoint from $G\cup G'$, thus contained in a closed regular 
neighborhood $N$ of $\Sigma$. After fixing a diffeomorphism between $N$ and $\Si\times [0,1]$ and 
consequently a projection $\pi: N \rightarrow \Sigma$, we may also assume that the restriction 
$\pi|_L\co L\to\Si$ is an immersion, and that each self-intersection of $\pi(L)\subset\Si$ is a transverse 
double point. Hence, we can represent $L$ using its diagram $D_L\subset\Si$ consisting of $\pi(L)$ together 
with the over-under information at each point of self-intersection. Recall that, for any framed, simple closed 
curve $\gamma\subset N$ contained in a horizontal copy of $\Si$, i.e.~a surface of the form 
$\Sigma \times \{ t \}\subset\Si\x [0,1]$, the relative framing  of $\gamma$ is the integer 
$\text{fr}(\gamma)-\text{fr}_{\Sigma \times \{t\}}(\gamma)$. Observe that a framed link 
$L=\cup_{i=1}^k L_i\subset N$ is horizontal if each connected component 
$L_i\subset L\subset N$ sits on a horizontal copy of $\Sigma$ and: 
\begin{enumerate}
\item[(1)]
the relative framing  of $L_i$ is $\pm 1$;
\item[(2)]
the diagram $D_{L_i}$ of $L_i$ has no crossings and is therefore a simple closed curve in $\Si$;
\item[(3)]
$D_{L_i}$ is non-separating in $\Si$;
\end{enumerate} 
moreover, for each $i<j$:
\begin{enumerate} 
\item[(4)]
at each crossing between $D_{L_i}$ and $D_{L_j}$ the overpassing arc belongs to 
$D_{L_j}$.
\end{enumerate}
We say that a crossing of $D_L$ is \textit{bad} if it involves a diagram $D_{L_i}$ which violates 
either $(2)$ or $(4)$. We are going to change the handle decomposition $\H$ into a horizontal 
decomposition of $X$ via a finite sequence of steps. At each step we shall either eliminate a bad 
crossing or adjust the framing of some $L_i$, making sure that 
the genus of $\Si$ is $u + g_H(\del_- X)$ at each step, and at the last step 
Conditions $(1)$ through $(4)$ above are satisfied.  
In the process of changing $\H$ we will add four-dimensional $(1,2)$-cancelling pairs  
using the following {\em stabilization} procedure for each pair:
\begin{itemize}
\item 
choose two disks $\Delta_1$ and $\Delta_2$ in $\Sigma$, disjoint from each other and from the projections of $L_1,...,L_n$, as well as an orientation-reversing diffeomorphism $\varphi$ between their boundaries;
\item 
let $B_i\subset \del^u_- X$, for $i=1,2$, be disjoint 3-balls such that $B_i\cap N$ corresponds to 
$\Delta_i \times [0,1]$ under the identification of $N$ with $\Si\times [0,1]$; 
\item 
attach a $1$-handle $D^3\x [0,1]$ to $\del^u_- X$ along $B_1$ and $B_2$, giving a cobordism 
from $\del^u_- X$ to $\del^{u+1}_- X$, with the latter viewed as the quotient of 
$\del^u_- X\setminus (B_1\cup B_2)$ obtained by gluing $\del B_1$ to $\del B_2$ 
via a diffeomorphism which looks like $\varphi\x\id_{[0,1]}$ when restricted to $N$;
\item 
attach the $2$-handle along a curve contained in a horizontal copy of $\Si$ inside $N$.
\end{itemize}
Note that the image $N'$ of $N\setminus (B_1\cup B_2)$ inside $\del_-^{u+1} X$ is diffeomorphic to 
$\Si'\x [0,1]$, where $\Si'$ is the quotient of $\Si\setminus (\De_1\cup\De_2)$ obtained by gluing 
$\del\De_1$ to $\del\De_2$ via the diffeomorphism $\varphi$. Also, we have a projection 
$\pi':N' \rightarrow \Sigma'$ induced by the restriction of $\pi$ to 
$N\setminus (B_1\cup B_2)$ and $\Si'\subset \del^{u+1}_- X$ is a Heegaard surface of genus 
$g(\Si)+1 = u + 1 + g_H(\del_- X)$. 

We are now going to start the modification of $\H$. Suppose that $i<j$ and let $c$ be a bad crossing 
between $D_{L_i}$ and $D_{L_j}$ inside $\Si$. Then, the over-passing arc at $c$ belongs to $L_i$, 
as in the left picture of Figure~\ref{f:stab}. 
\begin{figure}[ht]
\centering
\labellist
\hair 2pt
\pinlabel $L_i$ at 24 15
\pinlabel $L_i$ at 172 15
\pinlabel $L_i+L_{k+1}$ at 325 15
\pinlabel $L_j$ at 63 83
\pinlabel $L_j$ at 211 83
\pinlabel $L_j$ at 371 83
\pinlabel $L_{k+1}$ at 250 53
\pinlabel $L_{k+1}$ at 424 53
\endlabellist
\includegraphics[width=0.9\textwidth]{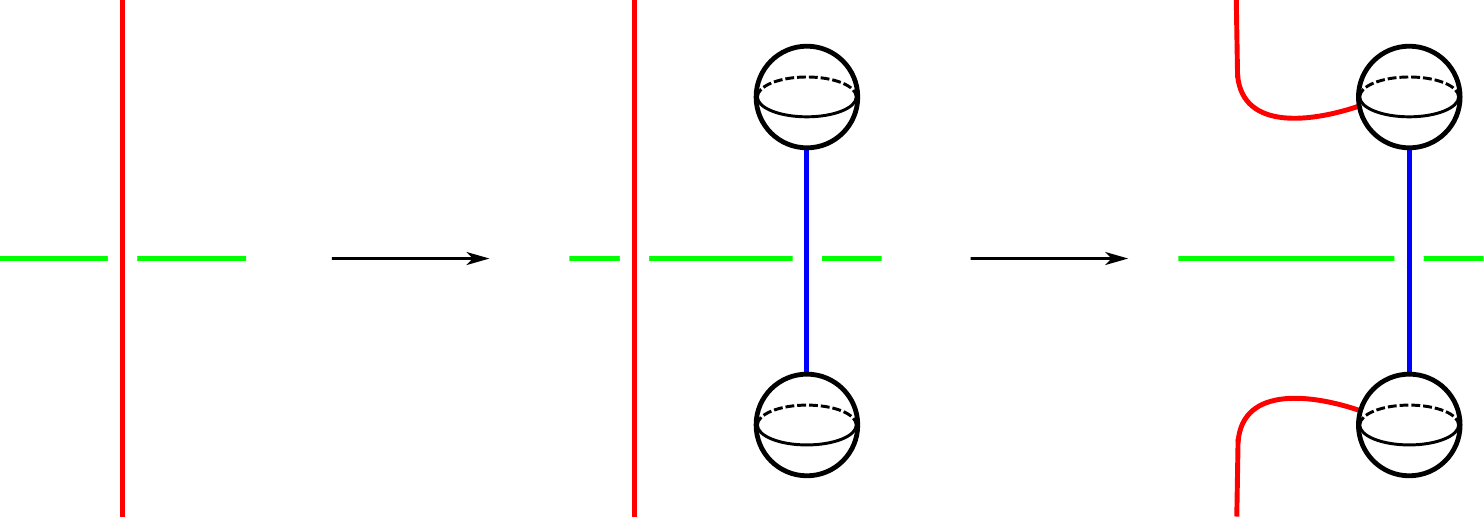}
\caption{Elimination of a bad crossing}
\label{f:stab}
\end{figure}
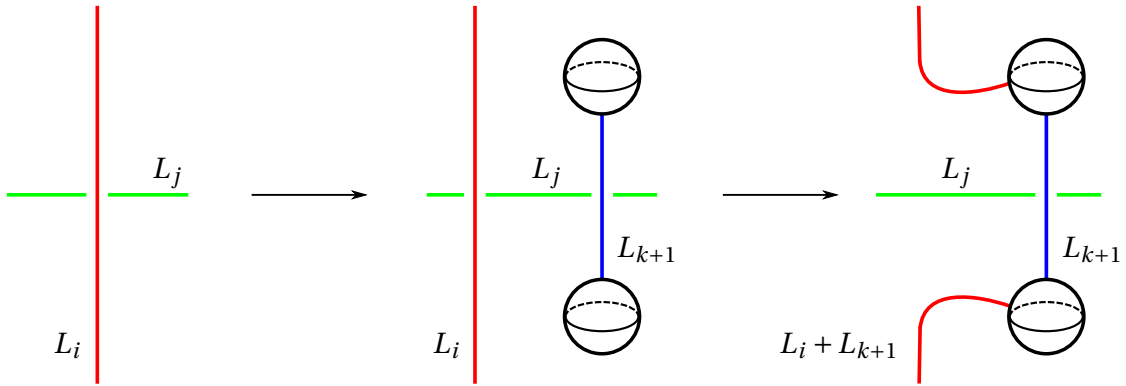
By a stabilization as above we introduce a cancelling pair near $c$, so that the attaching circle 
of the new $2$-handle, which we call $L_{k+1}$, has relative framing  $0$ and passes over $L_j$ 
once, as in the center picture of Figure~\ref{f:stab}. Sliding $L_i$ over $L_{k+1}$ we eliminate 
the bad crossing between $L_i$ and $L_j$, replacing it with a good crossing between $L_j$ and 
$L_{k+1}$, as in the right picture of Figure~\ref{f:stab}. The same argument works for self 
intersections, i.e.~when $i=j$. We can repeat this procedure until we get a diagram with no bad 
crossings, so that each component of $L$ can be assumed to sit on a horizontal copy of $\Si$.

Now we proceed to adjust the framings. In order to do that we use the Kirby calculus 
operation of twisting a $1$-handle by $360$ degrees, for which we refer 
to~\cite[Section~5.4]{GS99}. Figure~\ref{f:tw} illustrates the fact that 
a $360$-degree twist of one of the two attaching spheres of a $1$-handle changes 
the relative framings of the $2-$handles going over it. 
\begin{figure}[ht]
\centering
\labellist
\hair 2pt
\pinlabel $0$ at 23 70
\pinlabel $f$ at 23 155
\pinlabel $f$ at 152 156
\pinlabel $f$ at 287 156
\pinlabel $f-1$ at 430 156
\pinlabel $1$ at 420 70 
\endlabellist
\includegraphics[width=0.9\textwidth]{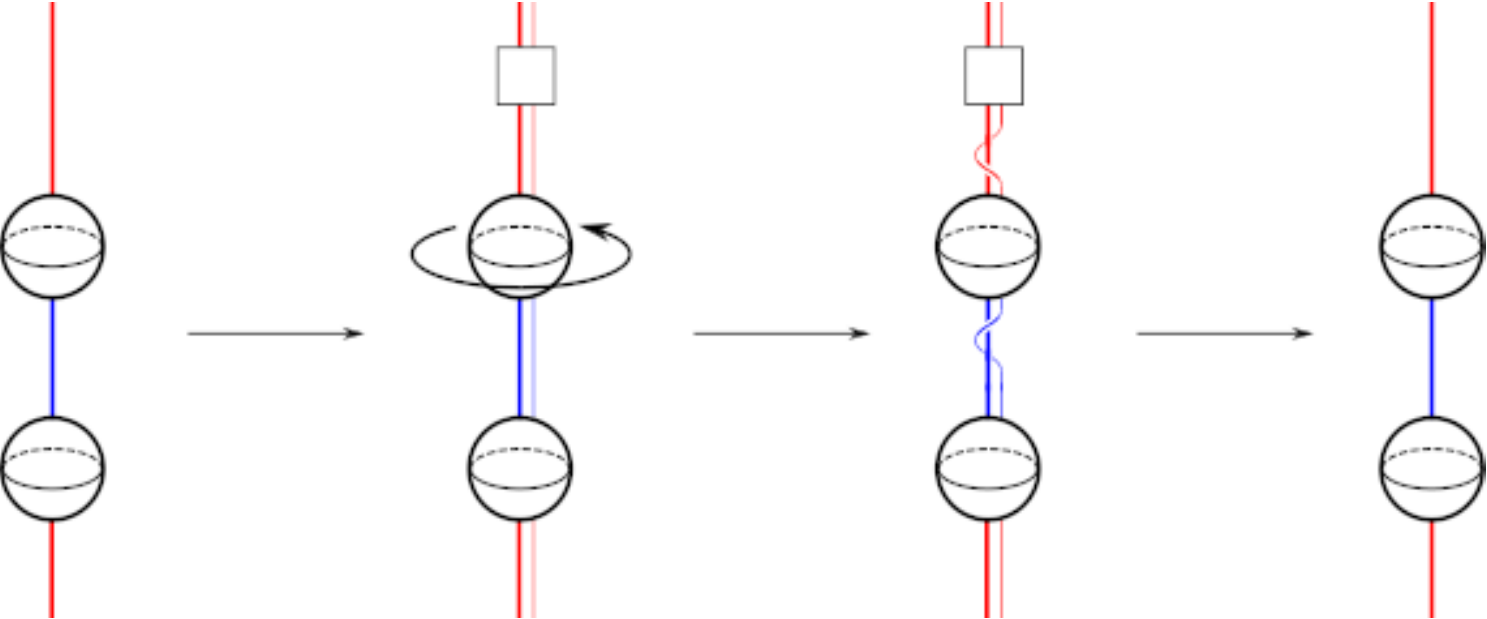}
\caption{Twisting a $1$-handle. The relative framing  of a knot can be represented by either an integer or a longitude: in the first and last step we are just switching from a notation to the other, while in the middle we are twisting a sphere as indicated by the arrow. Of course, the twist can also be performed in the opposite direction: in that case, the new framings would be $f+1$ and $-1$.}
\label{f:tw}
\end{figure}
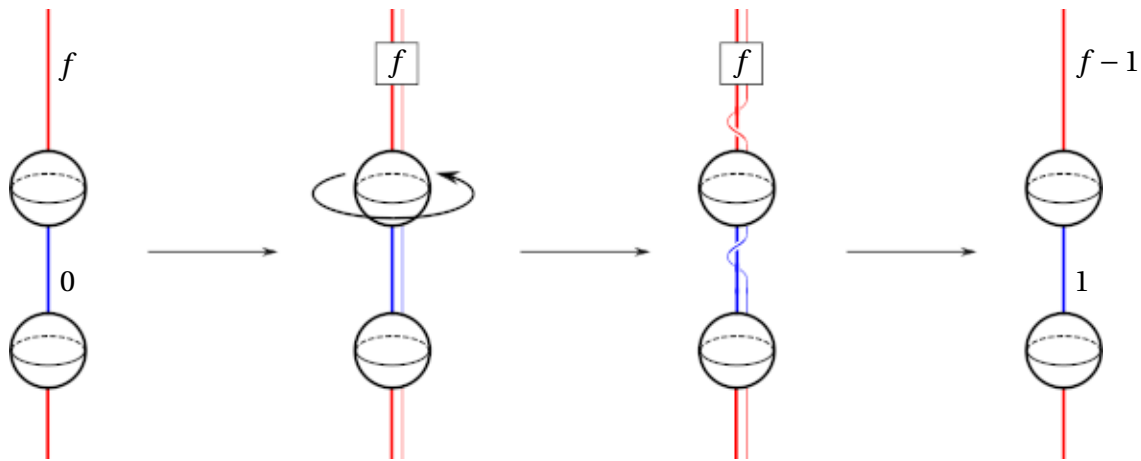
If we do this for all the $1$-handles of the cancelling pairs introduced to eliminate bad crossings, we can change the relative framings of the cancelling $2$-handles from $0$ to $\pm 1$. At that point, if $2$-handles with relative framing  not equal to $\pm 1$ are still present, we can first perform more 
stabilizations and handle slides as in Figure~\ref{f:stab} -- just imagine removing $L_j$ from the pictures. 
Then, apply twists with suitable signs as in Figure~\ref{f:tw}. Every time, we introduce a new $\pm 1$-framed 
$2$-handle and change by $\mp 1$ the relative framing  of an already existing $2$-handle. We can clearly keep 
going like this until all the relative framings are $\pm 1$, therefore so far we have ensured that $D_L$ 
satisfies properties $(1)$, $(2)$ and $(4)$ above. 
If a component of the diagram happens to be homologically trivial, we can make it homologically non-trivial 
with a single stabilization. In fact, consider the first two pictures from the left in Figure~\ref{f:stab} 
without the component $L_i$ and choosing the framing on $L_{k+1}$ so that 
its relative framing  is $\pm 1$. If $L_j$ is homologically trivial in the first picture, it is certainly not 
so in the second picture, and its relative framing  has not changed. This shows that after possibly performing 
a few more stabilizations Conditions $(1)-(4)$ can all be satisfied. 
Moreover, we made sure that the genus of 
$\Si$ is $u + g_H(\del_- X)$ at each step. This concludes the proof of Theorem~\ref{t:exist}. 
\medskip

As anticipated in Remark~\ref{r:tunnel}, we now sketch an alternative proof of the existence of 
handlebody decompositions suggested by the referee. The idea is based on a construction 
by B.~Clark~\cite{Cl80} showing that any link has finite tunnel number. Let $L\subset\del_-^u X$ 
be a framed link. As before, up to isotopy we may assume that $L$ sits in a 
neighborhood $N\cong\Si\x [0,1]$ of a Heegaard surface $\Si\subset\del_-^u X$ and  
that we have fixed a projection $\pi\co N\to \Si$ providing us with a diagram $D_L\subset\Si$. 
We may also assume that $D_L$ is in general position with respect to the set 
$\{\al_1,\ldots,\al_g\}\subset\Si$ of boundaries of a maximal set of disjoint compressing disks for 
the handlebody bounded by $\Si\x\{0\}$. For each crossing of $D_L$ and $D_L\cap(\cup_{i=1}^g\al_i)$ 
we add a ``vertical segment'' to $L\cup_{i=1}^g\al_i$, obtaining an embedded trivalent graph
$\Ga$ with the property that the complement in $\del_-^u X$ of a regular neighborhood $N(\Ga)$ 
is a handlebody. A pushed-off copy $L'$ of $L$ sits on the Heegaard surface $\del N(\Ga)$ in such a way 
that each connected component of $L'$ is homologically non-trivial, and its relative framing can be 
adjusted by an isotopy which adds to it an appropriate number of ``twists'' around $\Ga$. This shows that the 
the whole framed link $L$ sits horizontally on a Heegaard surface for $\del_-^u X$ having genus  
equal to $g$ plus the number of crossings of $D_L$ and $D_L\cap(\cup_{i=1}^g\al_i)$. It follows from 
this argument that every smooth cobordism as in Theorem~\ref{t:exist} admits a horizontal 
decomposition of type $(g,u,\ell,h)$ for some $g\geq u$. Note that, although this alternative proof gives a stronger conclusion on $L$, one loses control on the difference $g-u$. 

\section{Proof of theorem~\ref{t:type10} -- $\chi(X)=0$}\label{s:type10-chi=0}

We state and prove Lemma~\ref{l:heegaard} below in slightly greater generality than needed, for possible 
future use. The construction used in the lemma is essentially the same one used by Lickorish in 
his proof of the Lickorish--Wallace theorem~\cite{Li62}.

\begin{lemma}\label{l:heegaard}
Let $H_g \cup (\Sigma_g \times [0,1]) \cup H'_g$ be a Heegaard decomposition of a closed 
$3$-manifold $N$. Let $L \subset \Sigma_g \times [0,1]$ be a horizontal framed link 
with monodromy $m_L \in \Map(\Si_g)$ and associated cobordism $X_L\co \del_- X_L\to\del_+ X_L$. 
Let $\umu = (\mu_1,\ldots,\mu_g)$ and $\ula=(\la_1,\ldots,\la_g)$ be $g$-tuples of simple closed curves in 
$\Si_g$ bounding compressing discs in $H_g$, $H'_g$ respectively, 
so that $(\umu,\ula)$ is a Heegaard diagram for $N$. 
Then, $(m_L(\umu),\ula)$ is a Heegaard diagram for $\del_+ X_L$, where $m_L(\umu)$ denotes the 
$g$-tuple $(m_L(\mu_1),\ldots,m_L(\mu_g))$.
\end{lemma}

\begin{proof}
Given any orientation-preserving diffeomorphism $f\co\Si_g\to\Si_g$ we can can glue $\Si_g\x [0,1]$ to $H_g$ via $f$: 
\[
H_g\cup_f\left(\Si_g\x [0,1]\right) := H_g\cup \left(\Si_g\x [0,1]\right)/
\left(x\in\del H_g=\Si_g\sim f(x)\in\Si_g\x\{0\}\right).
\]
The result is diffeomorphic to $H_g$ via the map 
\[\
\psi = \id_{H_g}\cup (f^{-1}\x\id_{[0,1]})\co H_g\cup_f\left(\Si_g\x [0,1]\right) 
\to H_g\cup_{\id_{\Si_g}} \left(\Si_g\x [0,1]\right)\cong H_g.
\]
Let $\De_i\subset H_g$ be a compressing disk with $\del\De_i=\la_i$, 
$i=1,\ldots, g$. Then, 
\[
\del_+ X_L = H_g\cup_{m_L}\left(\Si_g\x [0,1]\right)\cup_{\id_{\Si_g}} H'_g
\]
and $\De'_i = \De_i\cup m_L(\mu_i)\x [0,1]$ is a compressing disk 
in $H_g\cup_{m_L}\left(\Si_g\x [0,1]\right)\cong H_g$ with 
$\del\De'_i = m_L(\mu_i)$. Therefore $(m_L(\umu),\la)$ is a Heegaard diagram 
for $\del_+ X_L$.
\end{proof}

The following proposition implies the first part of Theorem~\ref{t:type10}.
\begin{prop}\label{p:cob-type10} 
Let $X\co S^3\to \del_+ X$ be a smooth, oriented cobordism with a horizontal decomposition of type $(1,u,\ell,0)$, having Euler characteristic $\chi(X)=0$ and $b_1(\del_+ X) = 0$. Then, $u\in\{0,1\}$. 
If $u=0$ then $X\cong S^3\x [0,1]$, if $u=1$ 
and the monodromy is $\tau_\ga^\de$, $X$ 
is orientation-preserving diffeomorphic to 
$\pm B_{p,q}\setminus B^4$ for some $p>q\geq 0$, with the plus sign occurring if and only 
if $(p,q)=(1,0)$ or $\de = 1$. 
In particular, $\del_+ X = S^3$ if and only if $X$ is diffeomorphic to $S^3\x [0,1]$. 
\end{prop} 

\begin{proof} 
Since $g=1\geq u$ and $0=\chi(X) = \ell - u$, 
we have $\ell=u\in\{0,1\}$. If $u=\ell=0$ clearly $X\cong S^3\x [0,1]$. If $u=\ell=1$ 
the decomposition of $X$ 
consists of a $1$-handle and a $2$-handle attached along a simple closed curve 
$\ga\subset\del_+(S^1\x D^3\setminus B^4) = S^1\x S^2$  
sitting on the standard genus-$1$ Heegaard torus $T$. 
In this case the monodromy is simply $\tau_\ga^\de$ with $\de\in\{\pm 1\}$ and it 
coincides with its factorization. 
After choosing a suitable orientation we may assume that $\ga$ supports the homology class 
$p\mu+q\la\in H_1(T;\Z)$ with $p\geq 0$ and $p$, $q$ coprime, where $(\mu,\la)$  
is a pair of oriented, simple closed curves such that $\mu\cdot\la = 1$. Moreover, 
we can choose $\la$ as the boundary of a compressing disc 
in the solid torus $H_1$ and we are free to replace  
$\mu$ with $\mu + m\la$ for any $m\in\Z$. Thus, 
we may assume either $pq\neq 0$ or $(p,q)\in\{(1,0),(0,1)\}$. 
Clearly, if $(p,q) = (1,0)$ then $\ga = \mu$ and 
$X\cong B_{1,0}\setminus B^4 \cong S^3\x [0,1]$, while if $(p,q) = (0,1)$ then 
$\ga=\la$ and by Lemma~\ref{l:heegaard} $\del_+ X = S^1\x S^2$, 
therefore $b_1(\del_+ X)\neq 0$. 
Therefore, we may assume $pq\neq 0$. If we view $S^1\x S^2$ as $0$-surgery on an unknot in $S^3$, 
we can think of $\ga$ as sitting inside $S^3$. Then, the Seifert framing induced by $T$ on $\ga$ 
is equal to $\fr_T(\ga) = pq$. To see this it suffices to compute the linking number 
of $\ga$ with its push-off in the direction of the negative normal to $T$. 
Since $\ga = p\mu + q\la$, we have $\fr_T(\ga) = \lk(\ga,p\mu^- + q\la^-) = \lk(\ga,q\la^-) = pq$. 
Up to replacing $q$ with its remainder modulo $p$ and $\mu$ with $\mu$ plus a multiple of $\la$, 
we may assume $p>q>0$. If $X_L$ has factorization $(\tau_\ga^\de)$ the 
framing of $\ga$ prescribed by the $2$-handle is $\fr(\ga) = pq - \de$. 
By~\cite[Section~3.2]{LM14}, $X$ is orientation-preserving diffeomorphic to 
$B_{p,q}\setminus B^4$ if $\de =1$, $q>0$ and $\ga=p\mu-q\la$.  
Therefore, $\ga = p\mu + q\la = p(\mu+\la) - (p-q)\la$, so that 
$X$ is orientation-preserving diffeomorphic to 
$B_{p,p-q}\setminus B^4$. Note that   
taking the mirror image of the resulting Kirby calculus picture changes the sign of $\fr(\ga)$ and $\fr_T(\ga)$, 
therefore the sign of $\de$ as well. Thus, if $\de=-1$ 
after changing orientation the same argument shows that $-X$ is orientation-preserving diffeomorphic to 
$B_{p,q}\setminus B^4$, hence $X\cong -B_{p,q}\setminus B^4$. To conclude the proof it suffices to observe  
that $B_{p,q}$ is orientation-preserving diffeomorphic to $-B_{p,q}$ if and only if $p=1$. Indeed, the corresponding 
oriented boundaries $\del B_{p,q}=L(p^2,pq-1)$ and $\del(-B_{p,q})=L(p^2,p^2-pq+1)$ are orientation-preserving 
diffeomorphic if and only if $p^2-pq+1\equiv (pq-1)^{\pm 1}\bmod p^2$, which never holds if $p>1$. 
The last portion of the statement follows from the fact that $\del B_{p,p-q} = \del B_{p,q}$ is $S^3$ if and only if $B_{p,q}=B_{1,0}=B^4$. 
\end{proof} 

\section{Proof of Theorem~\ref{t:type10} -- $\chi(X)=1$}\label{s:type10-chi=1}

The proof is organized as follows. 
In Section~\ref{ss:firststep} we start the analysis of a smooth, oriented, 
cobordism $X\co S^3\to S^3$ having 
Euler characteristic $\chi(X)=1$. 
We show that each decomposition determines a triple of integers which is a solution 
of a Diophantine equation satisfying certain extra constraints. 
In Section~\ref{ss:secondstep} we completely determine 
the sets of all such triples. In Section~\ref{ss:hurwitz} we show that applying 
Hurwitz moves to the factorization of a horizontal framed link, the associated horizontal 
decomposition changes by a sequence of handle slides while staying horizontal. Finally, in 
Section~\ref{ss:laststep} we put everything together to finish the proof. 

\subsection{First step}\label{ss:firststep}  

In this section we start the analysis of a smooth, oriented, cobordism $X\co S^3\to S^3$  with Euler characteristic $\chi(X)=1$ 
having a horizontal decomposition of type $(1,u,\ell,0)$. Since the number $u$ of 1-handles satisfies $u\leq g=1$, there is either one or no 1-handles and no $3$-handles. 

If the decomposition contains no 1-handles then $\chi(X)=1$ implies that there is a single $2$-handle attached along a simple closed curve $\ga\subset S^3$ sitting on a genus-$1$ Heegaard torus $T$. 
Suppose that, with a suitable orientation, $\ga$ supports the homology class $p\mu+q\la\in H_1(T;\Z)$ 
with $p\geq 0$ and $p$, $q$ coprime, where $(\mu,\la)$ is a pair of oriented, 
simple closed curves such that $\mu\cdot\la = 1$. We can choose $\la$ as the boundary of a compressing disc 
in the solid torus $H_1$ and we are free to replace $\mu$ with $\mu + m\la$ for any $m\in\Z$. 
Therefore, we may assume either $p>q>0$ or $(p,q)\in\{(1,0),(0,1)\}$. Since the 2-handle is attached to $S^3$ with 
framing $\pm 1$ with respect to the framing $pq$ induced by $T$ and $H_1(S^3;\Z)=0$, we must have have $|pq\pm 1|=1$. 
The only possibilities are $(2,1)$, $(1,0)$ and $(0,1)$ and in every case we clearly have 
$X\cong \pm\CP^2\setminus\{B^4\cup B^4\}$. This concludes the proof when $\chi(X)=1$ if there are no 1-handles. 

From now on, we assume that the decomposition contains one 1-handle. 
In the notation of Section~\ref{s:intro}, the $2$-handles define a cobordism $X_L\co S^1\x S^2\to S^3$, where $L\subset S^1\x S^2$ 
is a link which is horizontal with respect to the standard genus-$1$ Heegaard decomposition. 
The assumption $\chi(X)=1$ implies that $L$ has two 
components $L_1, L_2$ sitting on parallel copies of the standard Heegaard torus $T \subset S^1\x S^2$. Let $N \cong T \times [0,1]$ 
be a regular neighborhood of $T$ containing $L$, so that $S^1\x S^2 = H_1 \cup N \cup H'_1$, and 
let $\pi\co N\to T$ be the projection map. Let $\ga_1 = \pi(L_1)$ and $\ga_2=\pi(L_2)$. By definition of horizontal decomposition, 
$\ga_1$ and $\ga_2$ are non-separating simple closed curves. Setting $\tau_i := \tau_{\ga_i}$, the factorization of $L$ is given by 
\begin{equation}\label{e:2factor}
F_L = (\tau_2^{\de_2},\tau_1^{\de_1}), 
\end{equation}
where $\de_i\in\{\pm 1\}$ is equal to minus the relative framing  of $L_i$ and 
the monodromy of $L$ is $m_L = \tau_2^{\de_2}\tau_1^{\de_1}$. 

Let $\la\subset T=\del H_1$ be a simple closed curve which bounds a disc in $H_1$.
We are going to show that, assuming the curves are oriented, the integers 
$\ga_1\cdot\la$, $\ga_2\cdot\la$ and $\ga_2\cdot\ga_1$ 
satisfy a certain Diophantine equation. In Section~\ref{ss:secondstep} we determine the solutions 
of the equation and in Section~\ref{ss:laststep} we use that knowledge together with the results of 
Section~\ref{ss:hurwitz} to finish the proof of Theorem~\ref{t:type10}.

We want to apply Lemma~\ref{l:heegaard} to our cobordism $X_L\co S^1\x S^2\to S^3$. As before, 
let $\la\subset T=\del H_1$ be 
the arbitrarily oriented boundary of a 
compressing disc in $H_1$. Then, $(\la,\la)$ is a Heegaard diagram for $S^1\x S^2$ and it follows from Lemma~\ref{l:heegaard} that 
$(m_L(\la),\la)$ is a Heegaard diagram for $\del_+ X_L = S^3$, which immediately implies 
\begin{equation}\label{e:inters}
|\la\cdot m_L(\la)| = 1.
\end{equation}  
From now on we assume $\ga_1$ and $\ga_2$ oriented and we 
abuse notation by denoting with $\ga_1$ and $\ga_2$ also the corresponding homology classes in $H_1(T;\Z)$. Define 
\[
x := \ga_1\cdot\la,\quad y := \ga_2\cdot\la\quad\text{and}\quad n := \ga_2\cdot\ga_1.
\]
Note that $\tau_1$ and $\tau_2$ do not depend on the orientations of $\ga_1$ and $\ga_2$. Since  
\[
m_L(\la) = \tau_2^{\de_2}(\tau_1^{\de_1}(\la)) = 
\tau_2^{\de_2}(\la + \de_1 x \ga_1) = 
\la + \de_1 x \ga_1 + \de_2 y \ga_2 + \de_1\de_2 n x \ga_2, 
\]
taking the intersection product of both sides with $\de_1\de_2\la$ yields  
\begin{equation}\label{e:quadratic} 
\de_2 x^2 + \de_1 y^2 + nxy = \ep,
\end{equation}
where $\ep = \de_1\de_2 m_L(\la)\cdot\la$ and by Condition~\eqref{e:inters} $|\ep|=1$.
Note that $\ep$ is independent of the orientations of $\la$, $\ga_1$ and $\ga_2$. 

\begin{defn}\label{d:S-set}
Define $S^{\de_2,\de_1}_{n,\ep}$ to be the set 
of pairs $(x,y)\in\Z^2$ such that $(x,y)$ is a solution of Equation~\eqref{e:quadratic} and satisfies the extra condition 
\begin{equation}\label{e:extra-cond}
n = ay - bx \quad\text{for some $a,b\in\Z$ with 
$\gcd(x,a)=\gcd(y,b)=1$}. 
\end{equation}
\end{defn}
Observe that a pair $(x,y)$ associated to a horizontal decomposition must satisfy 
Condition~\eqref{e:extra-cond}, because if $(\mu,\la)$ is a symplectic basis of $H_1(T;\Z)$ with $\mu\cdot\la = 1$ 
and we write $\ga_1=x \mu + a \la$, $\ga_2=y \mu + b \la$, we must have $\gcd(x,a)=\gcd(y,b)=1$ and $n=\ga_2 \cdot \ga_1=ay-bx$.

\begin{rmk}\label{r:equiv-extra-cond}
It easily follows from Equation~\eqref{e:quadratic} that $x$ and $y$ are coprime.  
Moreover, Condition~\eqref{e:extra-cond} is equivalent to $\gcd(x,n)=\gcd(y,n)=1$. 
Indeed,
\begin{itemize}
    \item if Condition~\eqref{e:extra-cond} holds then $\gcd(x,n)=\gcd(x,ay-bx)=\gcd(x,ay)=1$, 
    and similarly $\gcd(y,n)=1$;
    \item conversely, suppose $\gcd(x,n)=\gcd(y,n)=1$. By Bezout's theorem there are $a,b\in\Z$ with $n=ay-bx$. Then, 
    we have $\gcd(x,a)=1$, otherwise a prime factor of both $x$ and $a$ would also divide $n$, 
    contradicting $\gcd(x,n)=1$. Similarly, $\gcd(y,b)=1$.
\end{itemize}
\end{rmk}

We analyze the sets $S^{\de_2,\de_1}_{n,\ep}$  
in Section~\ref{ss:secondstep} to prove Theorem~\ref{t:type10} and in Section~\ref{s:disj-emb}
to prove Theorem~\ref{t:precise-disj-emb}.

\subsection{Second step}\label{ss:secondstep} 
In this section we determine the set $S^{\de_2,\de_1}_{n,\ep}$ 
of Definition~\ref{d:S-set}. 
Note that if $x$ and $y$ solve Equation~\eqref{e:quadratic} 
then they are are necessarily coprime.
Fix a a curve $\mu$ with $\mu\cdot\la = 1$. Then, we have 
\[
\ga_1 = p_1\mu + q_1\la,\quad \ga_2 = p_2\mu+q_2\la 
\]
for some $p_i, q_i$ with $(p_i,q_i)=1$ for $i=1,2$. Moreover, 
\begin{equation*}%\label{e:p1p2}
x = \ga_1\cdot\la = p_1,\quad y = \ga_2\cdot\la = p_2\quad\text{and}\quad 
n = \ga_2\cdot\ga_1 = p_2 q_1 - p_1 q_2.
\end{equation*}
\begin{rmk}\label{r:xy=0}
Observe that if $x=0$ then  Equation~\eqref{e:quadratic} implies $|y|=|p_2|=1$, and since 
$\ga_1=\pm\la$ we have $|q_1|=1$, therefore $|n|=|\ga_2\cdot\ga_1| = |p_2 q_1| = 1$. Similarly, 
$y=0$ implies $|n|=1$. 
\end{rmk}

Given $(x,y)\in S^{\de_2,\de_1}_{n,\ep}$, define 
\begin{equation}\label{e:mutations}
\hat x := -x-n\de_2 y\quad\text{and}\quad \hat y:=-y-n\de_1 x.
\end{equation}
We call the pairs $(\hat x, y)$ and $(x,\hat y)$ 
{\em mutations} of $(x,y)$. 
Moreover, let $m^{\de_2,\de_1}_{n,\ep}\in\N$ be the minimum of the function $S^{\de_2,\de_1}_{n,\ep}\to\N$ 
given by $(x,y)\mapsto |xy|$. We define the {\em bottom subset} 
$b(S^{\de_2,\de_1}_{n,\ep})\subset S^{\de_2,\de_1}_{n,\ep}$ as the set of 
pairs $(x,y)$ where the minimum $m^{\de_2,\de_1}_{n,\ep}$ is attained:  
\[
b(S^{\de_2,\de_1}_{n,\ep}) := 
\left\{(x,y)\in S^{\de_2,\de_1}_{n,\ep}\ |\ |xy| = m^{\de_2,\de_1}_{n,\ep}\right\}
\subset S^{\de_2,\de_1}_{n,\ep}.
\]
\begin{lemma}\label{l:mutations}  
If $(x,y)\in S_{n,\ep}^{\de_2,\de_1}$, both mutations of $(x,y)$ belong to $S_{n,\ep}^{\de_2,\de_1}$. Moreover, 
\begin{enumerate}
\item[(1)]
there is a finite sequence of mutations which sends $(x,y)$ into $b(S^{\de_2,\de_1}_{n,\ep})$; 
\item[(2)] 
if $S^{\de_2,\de_1}_{n,\ep}\not=\emptyset$ then $|n|\in\{1,3\}$ and 
$m^{\de_2,\de_1}_{n,\ep}=(|n|-1)/2$.
\end{enumerate}
\end{lemma} 

\begin{proof} 
The proof that $(\hat x,y)\in S_{n,\ep}^{\de_2,\de_1}$ 
is symmetric to the proof that $(x,\hat y)\in S_{n,\ep}^{\de_2,\de_1}$, so we 
provide the argument only for $(\hat x,y)$. One can easily check that 
if $(x,y)\in S_{n,\ep}^{\de_2,\de_1}$ then 
$(\hat x,y)$ is a solution of Equation~\eqref{e:quadratic}. 
Moreover, since $(x,y)\in S^{\de_2,\de_1}_{n,\ep}$, 
by Remark~\ref{r:equiv-extra-cond} we have $\gcd(x,n)=\gcd(y,n)=1$, 
which implies $\gcd(\hat x,n)=\gcd(-x-n \de_2 y, n) = \gcd(x,n)=1$, 
so that $(\hat x,y)\in S^{\de_2,\de_1}_{n,\ep}$ as well.
This proves the first part of the statement. 

Now observe that if $(x,y)\in S^{\de_2,\de_1}_{n,\ep}$ and $|x|=|y|$ then 
\[
x^2(\de_2 + \de_1 \pm n) = \ep,
\]
which implies $|x|=|y|=1$ and $|n|\in\{1,3\}$. Hence, if $(x,y)\in S^{\de_2,\de_1}_{n,\ep}$ 
and $|xy|>1$ we must have $|x|\neq |y|$. To prove $(1)$ we are going to use infinite descent to reduce to the case $|xy|\leq 1$. We claim that if $|x| > |y|$ then $|\hat x|<|x|$, while if $|x|<|y|$ then $|\hat y|<|y|$. Since the arguments in the two cases are similar, we only go through the case $|x| > |y|$. In this case it suffices to observe that if 
$|\hat x|\geq |x|$ then, using~\eqref{e:quadratic} we get  
\[
y^2 + 1 \geq |\de_1 \de_2 y^2-\de_2 \ep| = |x\hat x| \geq x^2 \geq (|y|+1)^2 = y^2 + 2|y| + 1, 
\]
which is impossible because $y\neq 0$. Therefore $|\hat x|<|x|$ and the claim is established. Inducting on $|xy|$ we can apply a sequence of mutations until we obtain a solution $(x_0,y_0)$ with $|x_0 y_0| \leq 1$. 
If $|x_0 y_0| = 0$ clearly $m^{\de_2,\de_1}_{n,\ep}=0$ and $(x_0,y_0)\in b(S_{n,\ep}^{\de_2,\de_1})$, and by Remark~\ref{r:xy=0} we have $(x_0,y_0)\in\pm\{(1,0),(0,1)\}$ and $|n|=1$, so that $m^{\de_2,\de_1}_{n,\ep}=(|n|-1)/2$.
If $|x_0 y_0| = 1$ then we saw above that $|n|\in\{1,3\}$. Moreover, if $|n|=3$ then 
$m^{\de_2,\de_1}_{n,\ep} \neq 0$ (again by Remark~\ref{r:xy=0}) and $|x_0y_0|=1$ implies $m^{\de_2,\de_1}_{n,\ep}=1=(|n|-1)/2$ and $(x_0,y_0)\in b(S_{n,\ep}^{\de_2,\de_1})$.
If $|n|=1$ then $\hat x_0=0 \Leftrightarrow n \de_2 y_0=-x_0 \Leftrightarrow \de_2=-nx_0y_0$ and, similarly, $\hat y_0=0 \Leftrightarrow \de_1=-nx_0y_0$; however, we cannot have $\de_1 = \de_2 = nx_0 y_0$ because it is inconsistent with 
Equation~\eqref{e:quadratic}, so that either $\hat x_0=0$ or $\hat y_0=0$, which implies that $m^{\de_2,\de_1}_{n,\ep}=0=(|n|-1)/2$ and that a further mutation sends $(x,y)$ into $b(S_{n,\ep}^{\de_2,\de_1})$. This completes the proof.
\end{proof}

Observe that the map $(x,y)\mapsto (x,-y)$ defines a bijection between $S_{n,\ep}^{\de_2,\de_1}$ and $S_{-n,\ep}^{\de_2,\de_1}$, while 
$S_{n,\ep}^{\de_2,\de_1}=S_{n,-\ep}^{-\de_2,-\de_1}$. Therefore, it will suffice to determine $S^{\de_2,\de_1}_{n,-1}$ for $n\in\{1,3\}$.

\begin{lemma}\label{l:bottomsets} 
The following hold: 
\begin{enumerate}
\item[(1)]
if $S^{\de_2,\de_1}_{3,-1}\not=\emptyset$ then $\de_1=\de_2=1$ and
$b(S^{1,1}_{3,-1}) = \pm\{(1,-1)\}$; 
\item[(2)]
if $S^{\de_2,\de_1}_{1,-1}\not=\emptyset$ then $\de_1+\de_2\in\{0,-2\}$; 
\item[(3)] 
$b(S^{-1,1}_{1,-1})= \pm\{(1,0)\}$;
\item[(4)]
$b(S^{1,-1}_{1,-1}) =  \pm\{(0,1)\}$; 
\item[(5)]
$b(S^{-1,-1}_{1,-1}) = \pm\{(1,0),(0,1)\}$.
\end{enumerate}
\end{lemma}

\begin{proof}
By Lemma~\ref{l:mutations}, if $S^{\de_2,\de_1}_{n,-1}\not=\emptyset$ then $b(S^{\de_2,\de_1}_{n,-1})\not=\emptyset$ and $m^{\de_2,\de_1}_{n,-1}=(|n|-1)/2$. 
If $n=3$ then $m^{\de_2,\de_1}_{3,-1}=1$, therefore there is 
$(x,y)\in S^{\de_2,\de_1}_{3,-1}$ with $|xy|=1$. Then, Equation~\eqref{e:quadratic} reads 
\[
\de_1 + \de_2 = -(1 + 3xy),
\]
which forces $xy=-1$, $\de_1=\de_2=1$ and 
$b(S^{1,1}_{3,-1}) = \pm\{(1,-1)\}$.
This proves $(1)$. If $n=1$ then $m^{\de_2,\de_1}_{1,-1}=0$ and there is 
$(x,y)\in S^{\de_2,\de_1}_{1,-1}$ with $|xy|=0$. Then, Equation~\eqref{e:quadratic} 
implies $(x,y)\in\{(\pm 1,0), (0,\pm 1)\}$ and either $\de_1=-1$ or $\de_2=-1$, therefore $(2)$ holds.
If $\de_1=\de_2=-1$ Equation~\eqref{e:quadratic} implies $x^2+y^2=1$ and $(5)$ holds. If $\de_1\de_2=-1$,  
either $(x,y)=\pm (0,1)$ and $(\de_1,\de_2)=(1,-1)$ or $(x,y)=\pm (1,0)$ and $(\de_1,\de_2)=(-1,1)$. 
Hence, $(3)$ and $(4)$ hold. 
\end{proof}

\subsection{Hurwitz moves and handle slides}\label{ss:hurwitz}  

In the last step of the proof of Theorem~\ref{t:type10}, Section~\ref{ss:laststep}, we will repeatedly 
use the following Proposition~\ref{p:hurwitz}, which shows that if one applies a Hurwitz move to 
the factorization of a horizontal framed link $L$, the associated horizontal decomposition changes 
by a sequence of handle slides while staying horizontal. This result is essentially an adaptation 
and an extension of~\cite[Lemma~5.1]{TY12}. In the following we use notation from Section~\ref{s:intro}.

\begin{prop}\label{p:hurwitz}
Let $L = \bigcup_{i=1}^k L_i \subset \Si_g\x [0,1]$ be a 
horizontal link with $L_i\subset\mathfrak{S}_i:=\Si_g\x\{t_i\}$ for $i=1,\ldots, k$.
Let $F_L = (\tau_k^{\de_k},\ldots,\tau_1^{\de_1})$ be the factorization of $L$. 
Define the horizontal link 
$L' = \bigcup_{i=1}^k L'_i$ by setting  
$L'_j = L_j$ as framed knots for $j\neq i, i+1$ and 
\begin{align*}
L'_i&:= L_{i+1}\x\{t_i\}, & \fr(L'_i)&:= \fr_{\mathfrak{S}_i}(L'_i) - \de_{i+1}, \\
L'_{i+1}&:= \tau_{i+1}^{\de_{i+1}}(\pi(L_i))\x\{t_{i+1}\}, 
& \fr(L'_{i+1})&:= \fr_{\mathfrak{S}_{i+1}}(L'_{i+1}) - \de_i.
\end{align*}
Then, the factorization $F_{L'}$ is obtained from $F_L$ by the Hurwitz move  
\[
F_L = (\ldots,\tau_{i+1}^{\de_{i+1}},\tau_i^{\de_i},\ldots)\to 
F_{L'} = (\ldots,\tau_{i+1}^{\de_{i+1}}\tau_i^{\de_i}\tau_{i+1}^{-\de_{i+1}},
\tau_{i+1}^{\de_{i+1}},\ldots).
\]
If we set, instead, $L'_j = L_j$ for $j\neq i, i+1$ and 
\begin{align*}
L'_i&:= \tau_i^{\de_i}(\pi(L_{i+1}))\x\{t_i\}, & \fr(L'_i)&:= \fr_{\mathfrak{S}_i}(L'_i) - \de_{i+1},\\
L'_{i+1}&:= L_i\x\{t_{i+1}\}, & \fr(L'_{i+1})&:= \fr_{\mathfrak{S}_{i+1}}(L'_{i+1}) - \de_i, 
\end{align*}
the factorization $F_{L'}$ is obtained from $F_L$ by the Hurwitz move 
\[
F_L = (\ldots,\tau_{i+1}^{\de_{i+1}},\tau_i^{\de_i},\ldots)\to 
F_{L'} = (\ldots,\tau_i^{\de_i},\tau_i^{-\de_i}\tau_{i+1}^{\de_{i+1}}\tau_i^{\de_i},\ldots).
\]
In both cases, let $X$ and, respectively, $X'$ be cobordisms having horizontal decompositions having 
associated horizontal links $L$ and, respectively, $L'$. Then, the handlebody decomposition of 
$X_{L'}$ is obtained from the one induced on $X_L$ by a sequence of handle slides. 
In particular, $X$ and $X'$ are orientation-preserving diffeomorphic.  
\end{prop}

\begin{proof}[Proof of Proposition~\ref{p:hurwitz}]
Given simple closed curves $a,b\subset\Si_g$, it is a well-known and easily checked fact that 
$\tau_{\tau_a(b)} = \tau_a\tau_b\tau_a^{-1}$. From this and the definition of $L'$ it 
follows immediately that the factorization of $L'$ is obtained from that of $L$ as 
described. Therefore, we only need to check that the horizontal decomposition 
of $X_{L'}$ is obtained from that of $X_L$ by a sequence of handle slides. It clearly suffices to 
show this for a link with two components, so we may assume $L = L_1\cup L_2$, with $L_i\subset\mathfrak{S}_i:=\Si_g\x\{t_i\}$, $i=1,2$ and $t_1<t_2$. 
Suppose first that $\fr(L_2) = \fr_{\mathfrak{S}_2}(L_2) - 1$. After an isotopy,  
the diagram of the projection of $L$ to $\Si_g$ will appear inside a disk 
$D\subset \Si_g$ as in Figure~\ref{f:LiLj}, and will have no crossings 
outside of $D$. 
\begin{figure}[ht]
	\labellist
	\hair 2pt
	\pinlabel $L_1$ at 12 100
	\pinlabel $L_2$ at 100 38
	\pinlabel $-1$ at 100 65
	\endlabellist
	\centering
	\includegraphics[scale=0.7]{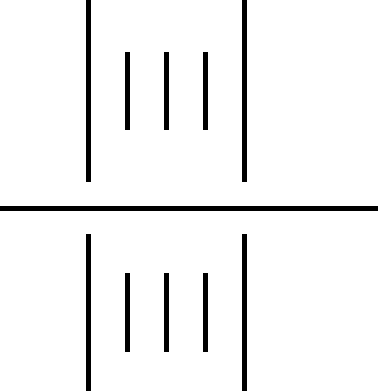}
	\caption{The local configuration of $\al$ and $\be$}
	\label{f:LiLj}
\end{figure}
Because of our assumption on the framings, indicated by the label $-1$ in the picture, the factorization of $L$ is given by $(\tau_{L_2}, \tau_{L_1}^{\pm 1})$. Assuming there are $c$ crossings in Figure~\ref{f:LiLj}, the left-hand side of Figure~\ref{f:slidev2} shows 
$c$ pushed-off copies of $L_2$ drawn in black. The link $L_1$ is shown in red, underneath $L_2$.  
\begin{figure}[ht]
	\labellist
	\hair 2pt
	\pinlabel $\lra$ at 270 157
	\pinlabel $L_1$ at 115 300
	\pinlabel $\sim \tau_{L_2}(\pi(L_1))\x\{t_2\}$ at 492 290
	\pinlabel $L_2$ at 210 140
	\pinlabel $-1$ at 20 140
	\pinlabel $L_2\x\{t_1\}$ at 535 140
	\pinlabel $-1$ at 370 140
	\endlabellist
	\centering
	\includegraphics[scale=0.6]{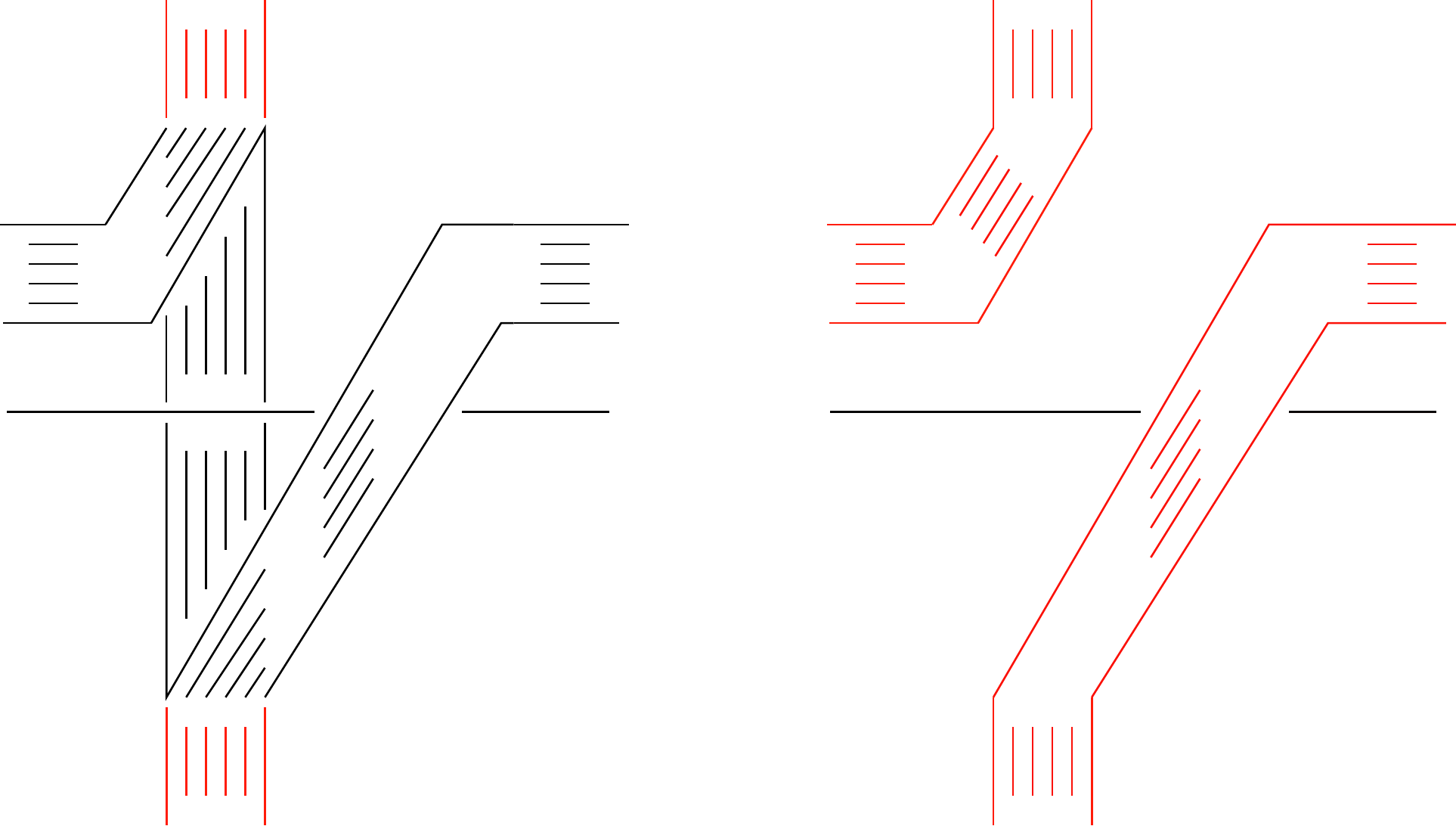}
	\caption{The result of $n$ handle slides}
	\label{f:slidev2}
\end{figure}
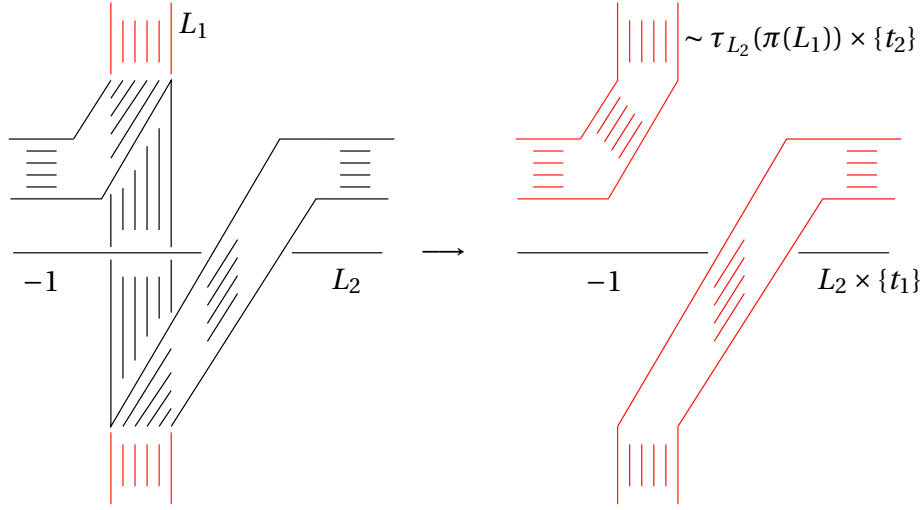
Sliding $c$ times the 2-handle attached along $L_1$ over the 2-handle attached along $L_2$
we obtain a curve of the form 
$L_1\setminus\{\cup_{i=1}^c I_1^{(i)}\}\cup L_2\setminus\{\cup_{i=1}^c I_2^{(i)}\}$, 
where $I_2^{(1)},\ldots,I_2^{(c)}\subset L_2$ are the $c$ vertical black segments visible 
on the left-hand side of Figure~\ref{f:slidev2} and $I_1^{(1)},\ldots,I_1^{(c)}\subset L_1$ the vertical 
red segments on $L_1$ directly underneath the $I_2^{(i)}$'s (not visible in the picture).
The right-hand side of Figure~\ref{f:slidev2} shows the resulting framed link, which is clearly 
framed isotopic to the horizontal framed link 
$L' = L_2\x\{t_1\}\cup \tau_{L_2}(\pi(L_1))\x\{t_2\}$, with $\fr(L_2\x\{t_1\}) = \fr_{\mathfrak{S}_1}(L_2\x\{t_1\}) - 1$ and 
\[
\fr(L_1\x\{t_2\}) = \fr_{\mathfrak{S}_2}(L_1\x\{t_2\}) - \fr_{\mathfrak{S}_1}(L_1) + \fr(L_1).
\]
The factorization of $L'$ is $(\tau_{L_2}\tau_{L_1}^{\pm 1}\tau_{L_2}^{-1},\tau_{L_2})$. 
This concludes the argument in the case $\fr(L_2) = \fr_{\mathfrak{S}_2}(L_2) - 1$. The case 
$\fr(L_2) = \fr_{\mathfrak{S}_2}(L_2) + 1$ can be treated similarly or reduced to the previous case by mirroring, 
therefore the proof is complete.   
\end{proof}  

As an illustration of Proposition~\ref{p:hurwitz} we show how to use Hurwitz moves to prove  
the existence of an orientation-preserving diffeomorphism between the 4-manifolds $X$ and $X'$ 
described by the pictures in Figure~\ref{f:diffeo}. The same diffeomorphism can of 
course be easily established also by Kirby calculus. 
\begin{figure}[ht]
	\labellist
	\hair 2pt
\pinlabel $1$ at 60 15
\pinlabel $-1$ at 120 15
\pinlabel $-1$ at 180 15
\pinlabel $-1$ at 190 115
\pinlabel $-1$ at 470 -5
\pinlabel $-1$ at 450 90
\pinlabel $-1$ at 490 130
\pinlabel $1$ at 515 145
	\endlabellist
	\centering
	\includegraphics[scale=0.6]{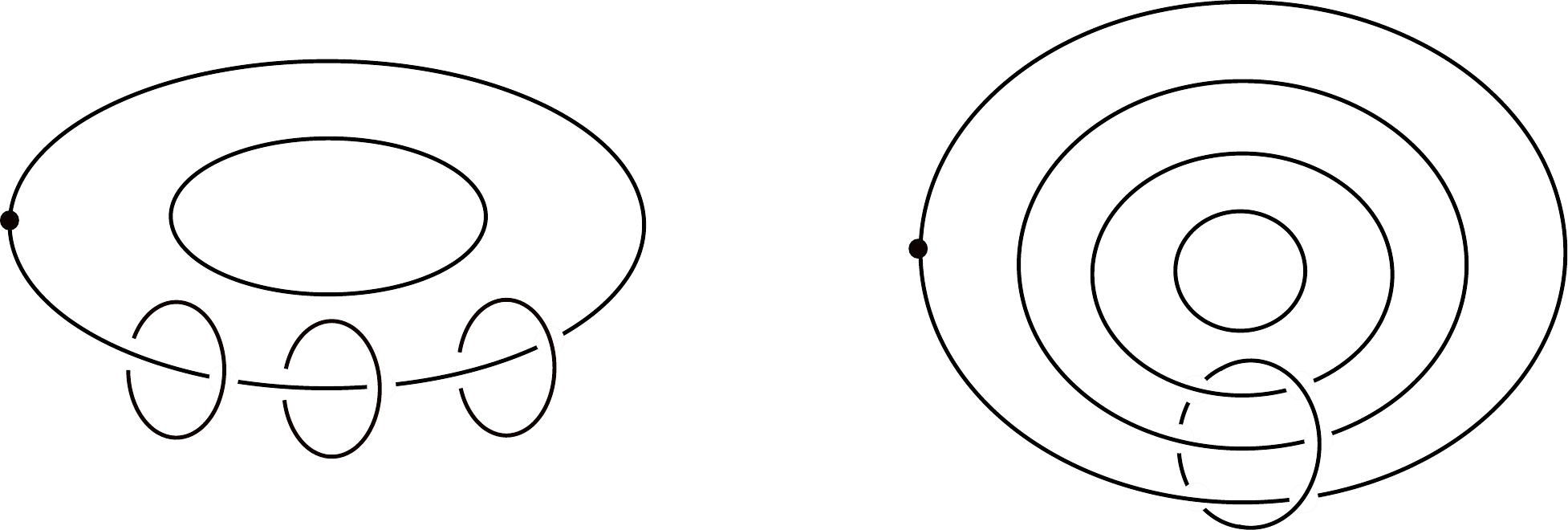}
	\vspace*{0.2cm}
	\caption{The diffeomorphic 4-manifolds $X$ and $X'$}
	\label{f:diffeo}
\end{figure}
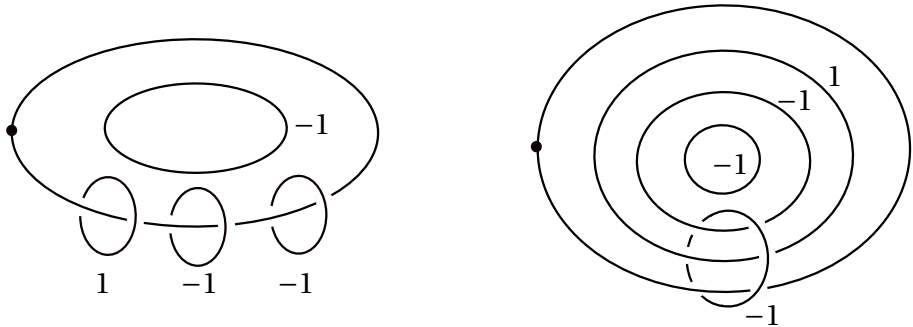
Both pictures of Figure~\ref{f:diffeo} can be viewed as describing 4-component horizontal framed links inside $S^2\x S^1$, 
giving $(1,1,3,0)$-decompositions of $X\setminus\left(B^4\sqcup B^4\right)$ and 
$X'\setminus\left(B^4\sqcup B^4\right)$ viewed as cobordisms 
$S^3\to S^3$. Let $\mu$ and $\la$, respectively, be the meridian and the Seifert longitude of the dotted unknot $U$. 
If we view $\mu$ and $\la$ as simple closed curves on the boundary of a standard neighborhood of $U$, 
then we see that the link $L$ on the left-hand side is horizontal with 
respect to the Heegaard decomposition described in Example~\ref{exa:exa1}, 
with factorization $F_L = (\tau_\la, \tau_\mu, \tau^{-1}_\mu, \tau_\mu)$. On the other hand, the link $L'$ on the right-hand side has 
factorization $F_{L'} = (\tau_\la,\tau_\mu,\tau_\la,\tau^{-1}_\la)$. We shall use the known fact that, given a 
diffeomorphism $f\co\Si_g\to\Si_g$ and a simple closed curve $a\subset\Si_g$, we have $\tau_{f(a)} = f\tau_a f^{-1}$. 
Denoting with $\stackrel{H}{\sim}$ Hurwitz equivalence, we have 
\[
F_L \stackrel{H}{\sim} 
(\tau_\mu,\tau_{\tau^{-1}_{\mu}(\la)},\tau^{-1}_\mu,\tau_\mu)\stackrel{H}{\sim}
(\tau_\mu,\tau^{-1}_\mu,\tau_\la,\tau_\mu)\stackrel{H}{\sim}
(\tau_\mu, \tau_\la, \tau^{-1}_{\tau^{-1}_\la(\mu)},\tau_\mu)\stackrel{H}{\sim}
(\tau_\mu,\tau_\la,\tau_\mu,\tau^{-1}_{\tau^{-1}_\mu \tau^{-1}_\la (\mu)})
\]
while 
\[
F_{L'} \stackrel{H}{\sim} 
(\tau_\mu,\tau_{\tau^{-1}_\mu(\la)},\tau_\la,\tau^{-1}_\la)\stackrel{H}{\sim}
(\tau_\mu,\tau_\la,\tau_{\tau^{-1}_\la \tau^{-1}_\mu (\la)},\tau^{-1}_\la) =
(\tau_\mu,\tau_\la,\tau_\mu,\tau^{-1}_{\tau^{-1}_\mu \tau^{-1}_\la (\mu)}),
\]
because $\tau^{-1}_\la \tau^{-1}_\mu (\la)$ is isotopic to $\mu$ and 
$\tau^{-1}_\mu \tau^{-1}_\la (\mu)$ to $-\la$. 
This shows $F_L\stackrel{H}{\sim} F_{L'}$, and therefore $X$ and $X'$ are diffeomorphic 
by Proposition~\ref{p:hurwitz}. 

\subsection{Last step}\label{ss:laststep}

In this section we conclude the proof of Theorem~\ref{t:type10}. We assume familiarity with the notation 
introduced in the previous sections. 
Let $X\co S^3\to S^3$ be the cobordism of Section~\ref{ss:firststep}, and $L$ the horizontal two-component link associated to a 
$(1,1,2,0)$-decomposition of $X$. 
We assume that $L$ has factorization $F_L=(\tau_2^{\de_2},\tau_1^{\de_1})$ 
for some $\de_1,\de_2\in\{\pm 1\}$. The oriented curves $\la,\ga_1,\ga_2\subset T$ determine
$(x,y)\in S_{n,\ep}^{\de_2,\de_1}$ for some $\ep\in\{\pm 1\}$. 

By Proposition~\ref{p:hurwitz} we know that applying Hurwitz moves to $F_L$ we obtain 
factorizations of links associated to $(1,1,2,0)$-decompositions of $X$. Therefore, our plan will be to show that 
via Hurwitz moves we can change $F_L$, and consequently $(x,y)$, 
until we get a pair $(x',y')$ with 
$|x'y'|$ minimal. If we succeed, to determine the diffeomorphism type of $X$ 
will suffice to do it when the pair $(x,y)$ is minimal. 
We start with a couple of observations showing that we may assume $n\in\{1,3\}$ and $\ep = -1$.

First observe that, since reversing the orientation of a component of $L$ changes the sign of $n$ without altering $X_L$, from 
now on we may (and will) assume without loss of generality $n>0$. 
In view of Lemma~\ref{l:mutations}, the number $n$ will be either $3$ or 
$1$. Next, observe that if we represent $S^1 \times S^2$ as a dotted unknot $U \subset S^3$, we may assume that $T$ is the 
boundary of a regular neighborhood of $U$. We may also assume that the curves $\mu$ and $\la$ of the symplectic basis used to 
define $x$ and $y$ are, respectively, a meridian and a Seifert longitude of $U$. Taking the union with $L$ gives a Kirby 
diagram for $\hat X := B^4\cup_{\del_- X} X$. Note that taking the mirror image of such a diagram amounts to replacing the 
curves $\ga_i=p_i \mu + q_i \la$ with $p_i \mu-q_i \lambda$ and changing the signs of their relative framings. The result is 
a diagram of a $(1,1,2,0)$-decomposition of $-X$ with associated pair $(-x,-y)$. 
If we also reverse the orientation of $\ga_2$ we 
obtain a bijection between $S_{n,\ep}^{\de_2,\de_1}$ and $S_{n,-\ep}^{-\de_2,-\de_1}$ given by 
$(x,y) \leftrightarrow (x,-y)$. 
This means that, up to reversing the orientation of $X$ we may assume $\ep = -1$.

Let $U,V\co S_{n,\ep}^{\de_2,\de_1}\to S_{n,\ep}^{\de_2,\de_1}$ be given 
by $U(x,y)=(\hat x,y)$ and $V(x,y)=(x,\hat y)$, where 
\[
\mat{\hat x\\y} = \mat{-1&-n\de_2\\0&1}\mat{x\\y}\quad\text{and}\quad 
\mat{x\\\hat y} = \mat{1&0\\-n\de_1&-1}\mat{x\\y}.
\]
We warn the reader that, to avoid heavy notation, we are writing $U$ 
instead of $U^{\de_2,\de_1}_{n,\ep}$, and similarly for $V$. 
Define 
$S\co S_{n,\ep}^{\de_2,\de_1}\to S_{n,\ep}^{\de_1,\de_2}$
by $S(x,y) = (y,x)$. Then, $U$,$V$ and $S$ act on $S^{1,1}_{3,-1}$ and 
$S^{-1,1}_{1,-1}\cup S^{1,-1}_{1,-1}$ 
and satisfy the relations $US=SV$ and $VS=SU$. 

We now check what happens to the quantities $x$, $y$ and $n$ when we apply a Hurwitz move to $F_L$. The two possible Hurwitz moves are: 
\[
(\tau_2^{\de_2},\tau_1^{\de_1})\to 
(\tau_1^{\de_1},\tau_1^{-\de_1}\tau_2^{\de_2}\tau_1^{\de_1}
= \tau_{-\ga_2-n\de_1\ga_1}^{\de_2})
\quad\text{and}\quad
(\tau_2^{\de_2},\tau_1^{\de_1})\to 
(\tau_2^{\de_2}\tau_1^{\de_1}\tau_2^{-\de_2}= \tau_{-\ga_1-n\de_2\ga_2}^{\de_1},\tau_2^{\de_2}).
\]
So we get, respectively, 
\begin{equation}\label{e:trans}
(x,y,n)\to (-y-n\de_1 x,x,n)=(\hat y,x,n)\quad\text{and}\quad 
(x,y,n)\to (y,-x-n\de_2 y,n) = (y,\hat x,n).
\end{equation}
Thus, a Hurwitz move transforms 
$F_L$ into $F_{L'}$, where the triple $(x',y',n')$ associated to $L'$ satisfies 
$n'=n$ and, using notation from Section~\ref{ss:secondstep}, either $(x',y') = SU(x,y)$ or $(x',y')=SV(x,y)$. 
In particular, if $(x,y) \in S_{n,-1}^{\de_2,\de_1}$ then $(x',y')\in S_{n,-1}^{\de_1,\de_2}$. 

\noindent
{\bf Case $n=3$:}
since $\ep = -1$, in view of Lemma~\ref{l:bottomsets}(1) we may assume that the 
pair $(x,y)$ associated to $L$ is in $S_{3,-1}^{1,1}$, which is preserved by both $SU$ and $SV$. 
By Lemma~\ref{l:bottomsets}, a sequence of 
Hurwitz moves transforms $F_L$ into   $F_{L'}$, where $L'$ is a 
horizontal framed link whose associated pair is $(1,-1)$ and by 
Proposition~\ref{p:hurwitz} we have $X_L\cong X_{L'}$. 
Thus, it suffices to show that if $n=3$ and the pair is $(1,-1)\in S_{3,-1}^{1,1}$ 
then $X$ is diffeomorphic to $\CP^2\setminus (B^4\cup B^4)$. From $x=\ga_1 \cdot \la= 1$ we deduce that 
$\ga_1=\mu+c \la$ for some $c\in\mathbb{Z}$. We can simplify further the handlebody decomposition as follows. 
Any self-diffeomorphism $T\to T$ of the form $\tau_\la^m$, $m\in\Z$, extends to a 
diffeomorphism $\bar\tau\co H_1\to H_1$ and therefore to 
$\bar\tau\x\id\co H_1\x [0,1]\to H_1\x [0,1]$. We can extend $\bar\tau\x\id$ over the two $2$-handles $h_1$ 
and $h_2$ attached along $\ga_1$ and $\ga_2$ to an orientation-preserving diffeomorphism 
\[
\tilde\tau\co (H_1\x [0,1])\cup h_1\cup h_2\stackrel{\cong}{\to} 
(H_1\x [0,1])\cup \tilde\tau(h_1)\cup \tilde\tau(h_2),
\]
where the $2$-handles $\tilde\tau(h_i)$ are attached along $L' =\tilde\tau(L)$. The factorization of $F_{L'}$ 
is obtained by conjugating the elements of $F_L$ by $\tau_\la^m$ and the triple associated to $L'$ is still 
$(x,y,n)$. This shows that we may assume $\ga_1=\mu$ without loss of generality. Now $y=\ga_2 \cdot \la=-1$ 
and $n=\ga_2 \cdot \ga_1=3$ force $\ga_2=-\mu-3\la$. The associated Kirby diagram is isotopic to 
the one given in Figure~\ref{f:n=3}. 
\begin{figure}[ht]
	\labellist
	\hair 2pt
	\pinlabel $2$ at 52 100
	\pinlabel $-1$ at 120 95
	\endlabellist
	\centering
	\includegraphics[scale=0.5]{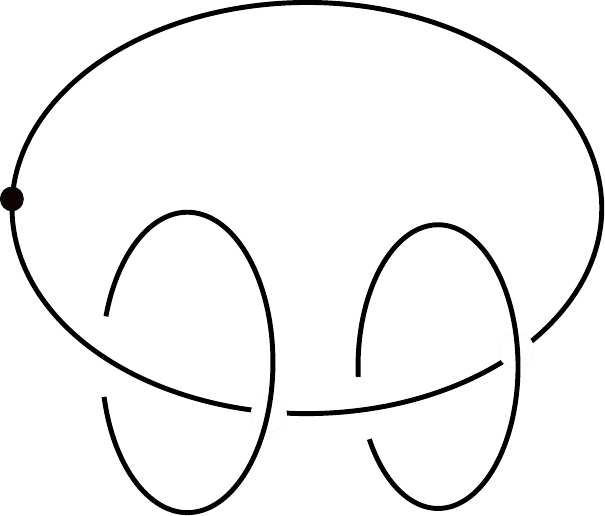}
	\caption{Kirby diagram in the case $n=3$}
	\label{f:n=3}
\end{figure}
It can be transformed into the standard diagram of $\CP^2$ by sliding the $2$-framed handle over 
the $-1$-framed handle and then 
cancelling the $1$-handle with the $-1$-framed $2$-handle.

\noindent
{\bf Case $n=1$:}
since $\ep = -1$, in view of Lemma~\ref{l:bottomsets} we may assume that the pair $(x,y)$ 
associated to $L$ is either in 
$S_{1,-1}^{1,-1}\cup S_{1,-1}^{-1,1}$ or $S_{1,-1}^{-1,-1}$. In the latter case a sequence of Hurwitz moves 
transforms $F_L$ into   $F_{L'}$, where $L'$ is a horizontal framed 
link whose associated pair $(x,y)\in \pm\{(1,0),(0,1)\}$ and by 
Proposition~\ref{p:hurwitz} we have $X_L\cong X_{L'}$. 
If $(x,y)$ is, up to sign, equal to $(1,0)$ then, since $\ga_2$ is non-separating it must 
be parallel to $\lambda$. Moreover, 
$x=\ga_1\cdot\la = 1$ and $n=1=\ga_2\cdot\ga_1$ imply $(\ga_1,\ga_2) = \pm (\mu, \la)$. 
Since $\de_1=\de_2=-1$, a Kirby calculus picture is given by a cancelling $(1,2)$-pair with $+1$-framed 
meridian $2$-handle together with an unlinked $+1$-framed unknot. 
Therefore in this case $X\cong\CP^2\setminus(B^4\cup B^4)$. If $(x,y)=(0,1)$ we deduce 
$(\ga_1,\ga_2) = \pm (\la,\mu)$ and reach the 
same conclusion as before. Now suppose 
$(x,y)\in S_{1,-1}^{1,-1}\cup S_{1,-1}^{-1,1}$. 
By Lemma~\ref{l:bottomsets}, a sequence of 
Hurwitz moves transforms $F_L$ into $F_{L'}$, where $L'$ is a horizontal framed link with associated 
pair equal to either $(\pm 1,0)\in S_{1,-1}^{-1,1}$ 
or $(0,\pm 1)\in S_{1,-1}^{1,-1}$.
In the first case we deduce as before $(\ga_1,\ga_2) = \pm (\mu,-\la)$, in the second case 
$(\ga_1,\ga_2) = \pm (\la,\mu)$. In both cases the associated Kirby diagram is easily shown to represent 
$\CP^2$. This concludes the proof of Theorem~\ref{t:type10}.

\begin{rmk}\label{r:epsilon}
The above proof shows that when the pair $(x,y)$ associated to the 
$(1,1,2,0)$-decomposition of $X$ and the oriented curves 
$\la,\ga_1$ and $\ga_2$ belongs to $S^{\de_2,\de_1}_{n,\ep}$ 
with $n\in\{1,3\}$ and $\ep=-1$, then $X$ is {\em orientation-preserving} diffeomorphic to $\CP^2\setminus(B^4\cup B^4)$. 
Since the quantity $\ep = \de_1\de_2 m_L(\la)\cdot\la$ is invariant under orientation changes of 
$\la,\ga_1$ and $\ga_2$ and Hurwitz moves, while it changes sign if the orientation of $X$ is 
reversed, we conclude that $\ep = -\si(X)$. This fact can be also established directly, applying  
Wall's non-additivity formula~\cite{Wa69}. 
\end{rmk}

\section{Proof of Theorem~\ref{t:disj-emb}}\label{s:disj-emb}

In this section we determine the smooth embeddings $B_{p,q}\subset\CP^2$ implicit in our proof of 
Theorem~\ref{t:type10}. Recall that, given a $(g,u,\ell,h)$-decomposition of an oriented cobordism 
$X\co\del_- X\to\del_+ X$, we denote by $X_L$ the cobordism determined by attaching the $2$-handles 
to the horizontal link $L\subset \del^u_- X$. More generally, if $L'\subset L$ is a sublink of $L$ 
we denote by $X_{L'}$ the cobordism determined by the corresponding $2$-handles. 
The following lemma is key to our arguments. 

\begin{lemma}\label{l:embed2balls}
Let $\ga_1\cup\ga_2\subset S^1\x S^2$ be a two-component link, 
horizontal with respect to the standard genus-$1$ Heegaard decomposition and 
associated to a $(1,1,2,0)$-decomposition of $X=\CP^2\setminus(B^4\cup B^4)$. Then, 
setting $\hat X_{\ga_i} := S^1\x D^3\cup_{\del_- X_{\ga_i}} X_{\ga_i}$, for $i=1,2$, 
the disjoint union $\hat X_{\ga_1}\sqcup\hat X_{\ga_i}$ admits a smooth orientation-preserving 
embedding in $\CP^2$.
\end{lemma}

\begin{proof}  
Each $2$-handle $h_i$, $i=1,2$, is attached along $\ga_i\x\{t_i\}\subset \Si_1\x\{t_i\}\subset S^1\x S^2$, 
with $0<t_1<t_2<1$. We may clearly assume that the attaching loci of the handles are disjoint. 
Hence, in these cases $\CP^2$ contains $S^1\x D^3 \cup X_{\ga_1\cup \ga_2}$. Moreover, 
$S^1\x D^3\approx H_1\x [0,1]$ and $\del(H_1\x\{t_i\}) = \Si_1\x\{t_i\}$, so  
$\CP^2$ contains $H_1\x [t_i-\ep,t_i+\ep]\cup X_{\ga_i}$ for $i=1,2$ and some 
small $\ep>0$. Therefore, the disjoint union $(S^1\x D^3\cup X_{\ga_1})\cup (S^1\x D^3\cup X_{\ga_2})$ 
admits a smooth orientation-preserving embedding in $\CP^2$. 
\end{proof} 

In view of Lemma~\ref{l:embed2balls}, to prove Theorem~\ref{t:disj-emb} we need to identify 
the $4$-manifolds $\hat X_{\ga_i}$, $i=1,2$. Note that  
$X_{\ga_i}\co S^3\to \del_+ X_{\ga_i}$ is a cobordism having $\chi(X_{\ga_i})=0$ and a 
$(1,1,1,0)$-decomposition. Therefore, if 
$b_1(\del_+ X_{\ga_i}) = 0$ and the factorization is $(\tau_2^{\de_2},\tau_1^{\de_1})$, 
it follows from Proposition~\ref{p:cob-type10} that $\hat X_{\ga_i}$ is diffeomorphic to either $B^4$ or 
$\de_i B_{p,q}$ for some $p>q>0$. The following result clearly implies Theorem~\ref{t:disj-emb}. 

\begin{thm}\label{t:precise-disj-emb}
Suppose that a horizontal decomposition of $\CP^2\setminus(B^4\cup B^4)$ has type 
$(1,1,2,0)$, factorization $(\tau_2^{\de_2},\tau_1^{\de_1})$ and that 
$b_1(\del\hat X_{\ga_1})=b_1(\del\hat X_{\ga_2})=0$. 
Then, the set $\{\hat X_{\ga_1},\hat X_{\ga_2}\}$ is one of the following: 
\begin{enumerate} 
\item[(1)]
$\{B_{m+1},B_m\}$ for some $m \geq 0$, where $B_m = B_{F_{2m-1},F_{2m-5}}$ if $\de_1=\de_2=1$;
\item[(2)] 
$\{B'_{m+1},B'_m\}$ for some $m \geq 0$, where $B'_m = (-1)^m B_{F_{m+1},F_{m}}$ if $\de_1+\de_2=0$;
\item[(3)]
$\{B^4\}$ 
if $\de_1=\de_2=-1$.
\end{enumerate} 
Moreover, in Case $(1)$ we have $|\ga_1\cdot\ga_2|=3$, 
in Cases $(2)$ and $(3)$ we have $|\ga_1\cdot\ga_2|=1$. 
\end{thm} 

For the proof of Theorem~\ref{t:precise-disj-emb} we need the following Lemma~\ref{l:n=13}. 
Note that the involutions $S$ given by $S(x,y)=(y,x)$ and minus the identity $-I$ generate 
$G$-actions on $S^{1,1}_{3,-1}$ and $S^{-1,1}_{1,-1}\cup S^{1,-1}_{1,-1}$, 
where $G=\langle S, -I\rangle\cong\Z/2Z\x\Z/2\Z$. We denote by $H=\{I,-I\}\cong\Z/2\Z$ the 
subgroup of $G$ generated by $-I$.

\begin{lemma}\label{l:n=13}
Let \{$F_m\}_{m\in\Z}\subset\Z$ be the Fibonacci sequence, 
given by $F_{-1} = 1$, $F_0 = 0$ and $F_{m+1} = F_{m} + F_{m-1}$. 
Then, we have 
\[
S_{3,-1}^{1,1} = G\cdot\S, \quad 
S_{1,-1}^{-1,1} = H\cdot(\T_1\cup\T_2),\quad
S_{1,-1}^{1,-1} = S\cdot S_{1,-1}^{-1,1}
\]
\[
\text{and}\quad S^{-1,-1}_{1,-1} = \pm\{(1,0),(0,1),(1,1)\}, 
\]
where 
\[
\S = \{(-F_{2k-1},F_{2k+1})\in\Z^2\ |\ k\geq 0\} 
\]
\[
\T_1 =  \{(F_{2k+1},F_{2k})\ |\ k\geq 0\} %= \{(-UV)^k(1,0)\ |\ k\geq 0\}
\quad\text{and}\quad  
\T_2 = \{(F_{2k+1},-F_{2k+2})\ |\ k\geq 0\} %= V\cdot\T_1.
\]
\end{lemma} 

\begin{proof}
{\bf First case:} $n=3$, $\de_1=\de_2=1$. 
By Lemmas~\ref{l:mutations} and~\ref{l:bottomsets}, 
each $(x,y)\in S_{3,-1}^{1,1}$ is sent to either 
$(1,-1)$ or $(-1,1)$ by a 
sequence of mutations. Then, $(x,y)$ is obtained from either $(1,-1)$ or $(-1,1)$ applying a map of the form 
$\cdots UVU$ or $\cdots VUV$. Since $S\co S_{3,-1}^{1,1}\to S_{3,-1}^{1,1}$ 
is an involution which intertwines the actions of $U$ and $V$ 
and $(1,-1)=S(-1,1)$, if $(x,y)\in S_{3,-1}^{1,1}$ then 
$(x,y)$ belongs to the $G$-orbit of either 
$(VU)^h(-1,1)$ or $U(VU)^h(-1,1)$ for some $h\geq 0$. 
Since  $VU = (VS)^2 = (-VS)^2$ and $U = S^2 V = S(VS)$, we can rephrase the 
condition on $(x,y)$ by saying that $(x,y)$ belongs to the $G$-orbit of $(-VS)^k(-1,1)$ for some $k\geq 0$. 
The map $-VS$ acts as multiplication by $\sm{0&-1\\1&3}$, therefore $(x,y)$ belongs 
to the $G$-orbit of $\{(x_k,y_k)\in\Z^2\ |\ k\geq 0\}$, where    
\[
\mat{x_k\\y_k} = \mat{0&-1\\1&3}^k\mat{-1\\1}.
\]
The relation $(x_{k+1},y_{k+1}) = (-YS)(x_k,y_k)$ is equivalent to  
$x_{k+1} = -y_k$ and $y_{k+1}=x_k+3y_k$, which imply 
$x_{k+2}=3x_{k+1}-x_k$. Since $(x_0,y_0)=(-1,1)$ and $(x_1,y_1)=(-1,2)$, 
we have $x_k = -F_{2k-1}$ and $y_k = F_{2k+1}$ for each $k\geq 0$, therefore 
$(x,y)$ belongs to $G\cdot\S$. This shows $S^{1,1}_{3,-1}\subset G\cdot\S$. 
Conversely, since $V,S$ and $-I$ preserve $S^{1,1}_{3,-1}$ and the latter 
contains $(-1,1)$, we have $G\cdot\S\subset S^{1,1}_{3,-1}$.

{\bf Second case:} $n=1$, $\de_1+\de_2=0$. 
By Lemmas~\ref{l:mutations} and~\ref{l:bottomsets}, a sequence of mutations sends 
each $(x,y)\in  S_{1,-1}^{-1,1}$ to $(1,0)$ or $(-1,0)$ and each $(x,y)\in S_{1,-1}^{1,-1}$ 
to $(0,1)$ or $(0,-1)$. Since $S\co S_{1,-1}^{-1,1}\to S_{1,-1}^{1,-1}$
defines a bijection intertwining $U$ and $V$ and commuting with $-I$, 
it will suffice to determine the elements of $S^{-1,1}_{1,-1}$. Since 
$U(1,0)=(-1,0)$, if $(x,y)\in S_{1,-1}^{-1,1}$ 
then $(x,y)$ belongs to the $H$-orbit of either $(UV)^h (1,0)$ or $V(UV)^h(1,0)$ for some 
$h\geq 0$. Moreover, 
\[
UV = -\mat{1&1\\1&0}^2
\]
and we have 
\[
(UV)^h = (-1)^h \mat{1&1\\1&0}^{2h},
\]
therefore $(UV)^h(1,0)  = (-1)^h(x_{2h},y_{2h})$, where the sequence $\{(x_k,y_k)\}$ is defined 
by $(x_0,y_0)=(1,0)$ and 
\[
\mat{x_{k+1}\\y_{k+1}} = \mat{1&1\\1&0}\mat{x_k\\y_k},\quad k\geq 0.
\]
But this is equivalent to $y_k = x_{k-1}$ and $x_{k+1} = x_k+x_{k-1}$, therefore $x_k = F_{k+1}$,  
where $F_k$ is the $k$-th Fibonacci number. We conclude  that $(UV)^h(1,0)  = (-1)^h(F_{2h+1},F_{2h})$. 
Since $V(F_{2h+1},F_{2h}) = (F_{2h+1}, -F_{2h+1}-F_{2h}) = (F_{2h+1},-F_{2h+2})\in S_{1,-1}^{-1,1}$, 
the proof of this case is finished. 
 
 {\bf Third case:} $n=1$, $\de_1=\de_2=-1$. 
 By Lemmas~\ref{l:mutations} and~\ref{l:bottomsets}, a sequence of mutations sends 
 each $(x,y)\in  S_{1,-1}^{-1,-1}$ to $(\pm 1,0)$ or $(0,\pm 1)$. But it is easy to check that 
 acting with $U$ and $V$ on $\pm (1,0)$ and $\pm (0,1)$ the only other pairs one can get 
 are $\pm (1,1)$. This concludes the proof. 
 \end{proof} 

\begin{proof}[Proof of Theorem~\ref{t:precise-disj-emb}]
In view of Proposition~\ref{p:cob-type10}, identifying the possible $X_{\ga_1}$ and $X_{\ga_2}$ amounts to 
identifying (up to sign) the possibile homology classes of the curves $\ga_1$ and $\ga_2$. Using the notation of 
Section~\ref{s:type10-chi=1}, let us fix a symplectic basis 
$(\mu,\la)$ of $H_1(T;\Z)$ with $\mu\cdot\la = 1$, so that $\ga_i = p_i\mu + q_i\la$
for some $p_i, q_i$ with $(p_i,q_i)=1$, $i=1,2$. Note that by Proposition~\ref{p:cob-type10} 
the assumption $b_1(\del\hat X_{\ga_i})=b_1(\del_+ X_{\ga_i})=0$ 
implies $p_1 p_2\neq 0$. 

We have already observed that the map $(x,y)\mapsto (x,-y)$ 
defines a bijection between $S_{n,\ep}^{\de_2,\de_1}$ and $S_{-n,\ep}^{\de_2,\de_1}$. 
Therefore, to list the possible $(|p_1|,|p_2|)$ it suffices to assume $n\in\{1,3\}$. 
Moreover, by Remark~\ref{r:epsilon} we have $\ep=-\si(X)=-1$. 
Then, by Lemma~\ref{l:n=13} each pair $(|p_1|,|p_2|)$ is one of the following: 
\begin{itemize}
\item 
$(F_{2k+1},F_{2k-1})$ or $(F_{2k-1},F_{2k+1})$ for some $k\geq 0$ when $\de_1=\de_2=1$ and $n=3$; 
\item 
$(F_{2k+1},F_{2k})$ or $(F_{2k+1},F_{2k+2})$ for some $k\geq 0$ when $(\de_1,\de_2)=(1,-1)$ and $n=1$;
\item 
$(F_{2k},F_{2k+1})$ or $(F_{2k+2},F_{2k+1})$ for some $k\geq 0$ when $(\de_1,\de_2)=(-1,1)$ and $n=1$.
\end{itemize}
Recall that the Fibonacci numbers satisfy Vajda's identity: 
\[
F_r F_{m+j} - F_m F_{r+j}  = (-1)^{r+1} F_{m-r} F_j.
\]
Choosing $j=4$, $m=2k-3$ and $r=2k-5$ we obtain the relation
\[
(\pm F_{2k-5}) F_{2k+1} -  (\pm F_{2k-3}) F_{2k-1} = \pm 3,
\]
which implies 
\begin{itemize}
\item 
$q_1\equiv \pm F_{2k-3}\bmod p_1$ and $q_2 \equiv \pm F_{2k-5}\bmod p_2$ when $(|p_1|,|p_2|) = (F_{2k+1},F_{2k-1})$;
\item 
$q_1\equiv \pm F_{2k-5}\bmod p_1$ and $q_2 \equiv \pm F_{2k-3}\bmod p_2$ when $(|p_1|,|p_2|) = (F_{2k-1},F_{2k+1})$. 
\end{itemize}
Since $B_{p,q}\cong B_{p,p-q}$, when $\de_1=\de_2=1$ and 
$|n|=3$ by Proposition~\ref{p:cob-type10} 
the 4-manifolds $\hat X_{\ga_1}$ and $\hat X_{\ga_2}$ are 
\[
B_{F_{2k+1},F_{2k-3}} = B_{k+1}\quad\text{and}\quad 
B_{F_{2k-1},F_{2k-5}} = B_k,\quad k\geq 0.
\]
Choosing $j=1$, $m=2k$ and $r=2k-1$ in Vajda's identity we obtain the relation
\[
(\pm F_{2k-1}) F_{2k+1}  - (\pm F_{2k}) F_{2k}  = \pm 1,
\]
which provides us with
\begin{itemize}
\item 
$q_1\equiv \pm F_{2k}\bmod p_1$ and $q_2 \equiv \pm F_{2k-1}\bmod p_2$ when $(|p_1|,|p_2|) = (F_{2k+1},F_{2k})$; 
\item 
$q_1\equiv \pm F_{2k-1}\bmod p_1$ and $q_2 \equiv \pm F_{2k}\bmod p_2$ when $(|p_1|,|p_2|) = (F_{2k},F_{2k+1})$.  
\end{itemize}
%in the first case when 
%$(\de_1,\de_2)=(-1,1)$ and $n=\pm 1$. 
In view of Proposition~\ref{p:cob-type10}, the corresponding rational balls are 
\[
B_{F_{2k+1},F_{2k}} = B'_{2k}\quad\text{and}\quad 
-B_{F_{2k},F_{2k-1}} = B'_{2k-1},\quad k > 0.
\]
Choosing $j=1$, $m=2k+1$ and $r=2k$ in Vajda's identity we obtain the relation
\[
(\pm F_{2k}) F_{2k+2}  - (\pm F_{2k+1}) F_{2k+1}  = \mp 1,
\]
which gives 
\begin{itemize}
\item 
$q_1\equiv \pm F_{2k}\bmod p_1$ and $q_2 \equiv \pm F_{2k+1}\bmod p_2$ when $(|p_1|,|p_2|) = (F_{2k+1},F_{2k+2})$;
\item 
$q_1\equiv \pm F_{2k+1}\bmod p_1$ and $q_2 \equiv \pm F_{2k}\bmod p_2$ when $(|p_1|,|p_2|) = (F_{2k+2},F_{2k+1})$.  
\end{itemize}
The rational balls are 
\[
B_{F_{2k+1},F_{2k}} = B'_{2k}\quad\text{and}\quad 
-B_{F_{2k+2},F_{2k+1}} = B'_{2k+1},\quad k\geq 0.
\]
Finally, if $(x,y) \in S_{1,-1}^{-1,-1}$ then by Lemma~\ref{l:n=13} we have 
$(p_1,p_2) = \pm (1,1)$ and the statement follows. 
This concludes the proof of Theorem~\ref{t:precise-disj-emb}. 
\end{proof} 

\printbibliography
\Addresses
\end{document}